\documentclass[11pt,draft,reqno]{amsart}
\usepackage{amsmath,amssymb,amscd,amsfonts,amsthm}
\usepackage{graphics}
\usepackage{dsfont}
\usepackage[all]{xy}
\usepackage{enumerate}

{\catcode`\@=11
\gdef\n@te#1#2{\leavevmode\vadjust{%
 {\setbox\z@\hbox to\z@{\strut#1}%
  \setbox\z@\hbox{\raise\dp\strutbox\box\z@}\ht\z@=\z@\dp\z@=\z@%
  #2\box\z@}}}
\gdef\leftnote#1{\n@te{\hss#1\quad}{}}
\gdef\rightnote#1{\n@te{\quad\kern-\leftskip#1\hss}{\moveright\hsize}}
\gdef\?{\FN@\qumark}
\gdef\qumark{\ifx\next"\DN@"##1"{\leftnote{\rm##1}}\else
 \DN@{\leftnote{\rm??}}\fi{\rm??}\next@}}
\DeclareFontFamily{OT1}{wncyr}{\hyphenchar\font45 }
\DeclareFontShape{OT1}{wncyr}{m}{n}{%
   <5> <6> <7> <8> <9> gen * wncyr
   <10> <10.95> <12> <14.4> <17.28> <20.74>  <24.88>wncyr10}{}
\DeclareFontShape{OT1}{wncyr}{m}{it}{%
   <5> <6> <7> <8> <9> gen * wncyi
   <10> <10.95> <12> <14.4> <17.28> <20.74> <24.88> wncyi10}{}
\DeclareFontShape{OT1}{wncyr}{m}{sc}{%
   <5> <6> <7> <8> <9> <10> <10.95> <12> <14.4>
   <17.28> <20.74> <24.88>wncysc10}{}
\DeclareFontShape{OT1}{wncyr}{b}{n}{%
   <5> <6> <7> <8> <9> gen * wncyb
   <10> <10.95> <12> <14.4> <17.28> <20.74> <24.88>wncyb10}{}
\input cyracc.def
\def\rus{\usefont{OT1}{wncyr}{m}{n}\cyracc\fontsize{9}{11pt}\selectfont}
\def\rusit{\usefont{OT1}{wncyr}{m}{it}\cyracc\fontsize{9}{11pt}\selectfont}
\def\russc{\usefont{OT1}{wncyr}{m}{sc}\cyracc\fontsize{9}{11pt}\selectfont}

\DeclareMathSizes{9}{9}{7}{5} 

\theoremstyle{plain}

\newtheorem{theorem}{Theorem}

\newtheorem{lemma}{Lemma}
\newtheorem{remark}{\it Remark}
\newtheorem{corollary}{Corollary}

\theoremstyle{definition}

\newtheorem{nothing*}[theorem]{}
\newtheorem{subnothing*}[sub]{}
\newtheorem{example}{Example}
\newtheorem{examples}[theorem]{Examples}

\newtheorem{question}[theorem]{Question}

\theoremstyle{remark}

\def\An{{\bf A}^n}

\def\bA2{{\mathbf A}\!^2}

\def\Aut {{\rm Aut\,}}
\def\Autn{{\rm Aut}\,\mathbf A\!^{n}}
\def\Autsn{{\rm Aut}^*\mathbf A\!^{n}}
\def\An{{\mathbf A}\!^n}
\def\GL{{\rm GL}}
\def\SL{{\rm SL}}
\def\Cr{{\rm Cr}}
\def\Aff{{\rm Aff}}

\begin{document}



\title[Tori in the Cremona groups]{Tori in the Cremona groups}

\author[Vladimir  L. Popov]{Vladimir  L. Popov${}^*$}
\address{Steklov Mathematical Institute,
Russian Academy of Sciences, Gubkina 8, Moscow\\
119991, Russia} \email{popovvl@mi.ras.ru}

\thanks{
 ${}^*$\,Supported by
 grants {\rus RFFI
11-01-00185-a}, {\rus N{SH}--5139.2012.1}, and the
program {\it Contemporary Problems of Theoretical
Mathematics} of the Russian Academy of Sciences, Branch
of Mathematics. }




\maketitle


\begin{abstract} We classify up to conjugacy the subgroups of certain types in the full, in the affine, and in the special affine Cremona groups.
We prove that the normalizers of these subgroups are algebraic.\;As an application, we obtain new results in the Linearization Problem genera\-liz\-ing to disconnected groups Bia{\l}ynicki-Birula's results  of 1966--67.
 We prove ``fusion theorems'' for  $n$-dimensional tori in the affine and in the special affine Cremona groups of rank  $n$. In the final section we introduce and discuss  the notions of Jordan decomposition and torsion primes for the Cremona groups.
\end{abstract}

\tableofcontents


\section{Introduction}

This work arose from the attempt to solve the problem posed in
\cite{Popov1}, \cite{Popov2}. In these papers have been introduced the notions of root
$\alpha$ and root vector  $D$ of an affine algebraic variety $X$ with respect to an algebraic torus  $T\subseteq {\rm Aut}\,X$. Namely, $D$ is a locally-nilpotent derivation of the coordinate algebra of the variety
 $X$, and $\alpha$ is a character of the torus $T$ such that
  $t^*\circ D\circ {t^*}^{-1}=\alpha(t)D$ for all $t\in T$.
  These definitions are inspired by the natural analogy with the classical definitions of the theory of algebraic groups and purport an attempt to apply to, in general, infinite dimensional group ${\rm Aut}\,X$ the techniques important in the theory of usual algebraic groups\footnote{In \cite{Popov1}, \cite{Popov2} was considered the case where  $X=\An$ and $T$
  is the maximal diagonal torus that leaves fixed the standard volume form, but
  this restriction plays no role in the definition.}.

 In \cite{Popov1}, \cite{Popov2} the following two problems related to the classical case where $X=\An$ and $T=D_{n}^*$ is the maximal diagonal torus  that leaves fixed the standard volume form  (see below \eqref{Dnnn}) have been posed:

\begin{enumerate}[\hskip 9.2mm]
\item[(R)] Find all the roots and root vectors of the variety
$\An$ with respect to $D_{n}^*$.
\item[(W)] Describe the normalizer and centralizer of the torus $D_n^*$
in the group ${\rm Aut}^*\An$ of automorphisms of the space $\An$ that leave fixed the stan\-dard volume form.
\end{enumerate}

 Problem (R) has been solved by  Liendo in \cite{Liendo}.\;The answer is the following. Let $x_1,\ldots, x_n$ be the standard coordinate functions on $\An$ and let  $\varepsilon_1,\ldots,\varepsilon_n$ be the
 ``coordinate'' characters of the standard $n$-dimensional diagonal torus  $D_n$ in $\Autn$ (see below \eqref{torusDn} and \eqref{epsilonn}). Then, up to multiplication by a nonzero constant, the root vectors are precisely all the derivations
 $D$ of the form
\begin{equation}\label{differ}
x_1^{l_1}\cdots x_n^{l_n}
({\partial}/{\partial x_i}),
\end{equation}
  where $l_1,\ldots, l_n$
  are nonnegative integers and  $l_i=0$. The root  $\alpha$ corresponding to root vector
 \eqref{differ}
 is the restriction to  $D_n^*$ of the character
\begin{equation*}
\varepsilon_i^{-1}\textstyle\prod_{j=1}^n\varepsilon_j^{l_j}.
\end{equation*}

 The problem mentioned in the beginning of this introduction is Problem (W). It is clear that it is aimed at getting a description of the ``Weyl group'' of the root system from Problem (R). We solve it in the present paper. Namely, we prove (Theorem \ref{12}) that the normalizer (centralizer) of the torus  $D_n^*$ in ${\rm Aut}^*\An$ coincides with its normalizer (centralizer) in ${\rm SL}_n$, so that the Weyl group of  $D_n^*$ in ${\rm Aut}^*\An$
 is the same as that of
  $D_n^*$ in ${\rm SL}_n$\,---\,it is the group of all permutations of the characters $\varepsilon_1,\ldots, \varepsilon_n$.

In fact, this result is only one special case of the series of general results that we obtain here. Namely, $D_n^*$
is only one of the infinitely many nonconju\-ga\-te diagonalizable algebraic subgroups
 $G$ of dimension $\geqslant n-1$ in the group $\Autn$.
 We prove that the normalizer of $G$ in $\Autn$ always is an  {\it algebraic} subgroup in  $\Autn$ (Theorem \ref{algnorm}).  It is the characteristic property of the specified dimensions: in general, for the diagonalizable subgroups of dimension
 $\leqslant n-2$ it does not hold. Moreover, in the case when nonconstant $G$-invariant polynomial functions on $\An$ exist, we explicitly describe the normalizer of  $G$ in $\Autn$, in particular, we show that in all but one  cases
 it coincides with the normalizer of  $G$ in
a group conjugate  to  $\GL_n$ (Theorem\,\ref{normal}).

Using the  found information, we obtain the new results in the Linearization Problem. In
 1966--67  Bia{\l}ynicki-Birula proved  \cite{BB1}, \cite{BB2} that every algebraic action on  $\An$  of an algebraic torus of dimension  $\geqslant n-1$ is equivalent to a linear action. We extend this statement to disconnected groups proving that every algebraic action
 on $\An$ of either an $n$-dimensional algebraic group whose connected component of identity is a torus, or an  $(n-1)$-dimensional diagonalizable group is equivalent to a linear action  (Theorems \ref{nnnnnn} and \ref{n-1}).

We also obtain the following classifications:

\begin{enumerate}[\hskip 2.2mm \rm(i)]
\item the classification of diagonalizable subgroups of the group  ${\rm Aff}_n$ of affine transformastions (see below \eqref{affi}) up to conjugacy in the full Cremona group  $\Cr_n={\rm Bir}\,\An$ (Theorem \ref{dc});
    \item the classification  of $n$-dimensional diagonalizable subgroups of  $\Autn$ up to conjugacy in $\Autn$ (Theorem \ref{nn});
        \item the classification of $(n-1)$-dimensional diagonalizable subgroups of  $\Autsn$
up to conjugacy in  $\Autsn$ (Theorem \ref{n-1*});
 \item the classiciation, up to conjugacy in  $\Autn$, of maximal   $n$-dimensio\-nal algebraic subgroups $G$ in $\Autn$ such that
$G^0$ is a torus (Theorem\,\ref{nn});
\item the classiciation, up to conjugacy in  $\Autsn$, of maximal   $(n-1)$-dimensional algebraic subgroups $G$ in $\Autsn$ such that
$G^0$ is a torus (Theorem \ref{12});
\item the classification of $(n-1)$-dimensional diagonalizable subgroups of  $\Autn$
up to conjugacy in с  $\Autn$ (Theorem \ref{n-1});
    \item  the classification of diagonalizable subgroups in
         $\Autn$ of dimension $\geqslant n-1$ up to conjugacy in $\Cr_n$ (Theorems \ref{nn} and \ref{Crn-11});
        \item the classification of one-dimensional tori of ${\rm Aut}\,\mathbf A^3$ up to conjugacy in ${\rm Aut}\,\mathbf A^3$ (Theorem \ref{1-3}).
\end{enumerate}

For instance, we prove that the set of classes of  $(n-1)$-dimensional diagonalizable subgroups of $\Autn$ that are conjugate in $\Autn$ is bijecti\-vely para\-metrized by the set of nonzero nondecreasing sequences
\begin{equation}\label{seq}
(l_1, \ldots, l_n)\in \mathbf Z^n,
\end{equation}
 such that $(l_1\ldots, l_n)\leqslant (-l_n\ldots, -l_1)$ with respect to the lexicographic order. Under this parametrization, to sequence  \eqref{seq} corresponds the
 class of the subgroup ${\rm ker}\,\varepsilon_1^{l_1}\cdots \varepsilon_n^{l_n}$.

Another example: we show that diagonalizable subgroups of
 ${\rm Aff}_n$ are conjugate in  $\Cr_n$ if and only if they are isomorphic and we
specify their canonical representatives. In particular (see Corollary
\ref{isoconj}), every two isomorphic finite Abelian subgroups of ${\rm Aff}_n$ are conjugate in  $\Cr_n$  (for finite cyclic subgroups this has been proved in \cite{Blanc06}).

In \cite{Serre2} Serre proved ``fusion theorem'' for the torus $D_n$ in $\Cr_n$.\;We prove and use ``fusion theorems'' for $n$-dimensional tori in  $\Autn$ and in
${\rm Aut}^*{\mathbf A}^{n+1}$ (Theorem\;\ref{FFF}).

In the final section, developing further the theme of analogies between the Cremona groups
and algebraic groups, we introduce and discuss the notions of Jordan decomposition and torsion primes  for the Cremona groups. In the course of discussion, we formulate some open ques\-tions.

\vskip 2mm

{\it Acknowledgment.} I am grateful to J.-P. Serre for the comments.

\newpage
{\bf Notation and conventions.}

\vskip 1mm

In the sequel, ``variety'' means ``algebraic variety over the fixed algebra\-i\-cal\-ly closed field $k$  of characteristic zero" in the sense of Serre's FAC \cite{Serre0}.   Apart from the standard notation and conventions of  \cite{Borel} and \cite{PV} used without reminders we also use the following:
\begin{enumerate}[\hskip 5mm ---]
\item ${\rm Mat}_{m\times n}(R)$ is the set of all matrices with $m$ rows, $n$  columns, and the coefficients in  $R$.

    \item $N_H(S)$ and $Z_H(S)$ are, respectively, the normalizer and centralizer of the subgroup $S$ of the group $H$.

         \item $\boldsymbol\mu_d$ is the subgroups of order $d$ in ${\mathbf G}_{\rm m}$.

        \item ${\rm X}(D)$ is the group of rational characters of the diagonalizable algebraic group $D$.

             \item $\chi(X)$ is the Euler characteristic of the variety  $X$ relative to the $l$-adic cohomology (for $k=\mathbf C$, by \cite{La} it coincides with the Euler characteristic relative to the usual cohomology with compact supports (cf.\;also \cite[Appendix]{KP})).

         \item If the group $G$ acts on a set $M$ and $\varphi\colon G\times M\to M$ is the map defining this action, then for the subsets $S\subseteq G$ and $X\subseteq M$, the subset $\varphi(S\times X)\subseteq M$ is denoted by $S\cdot X$
             (each time it is clear from the context what $\varphi$ is meant). In particular,
by $G\cdot a$ is denoted the $G$-orbit of the point $a$. By $G_a$ is denoted the $G$-stabilizer of the point\;$a$.

\item $x_1,\ldots, x_n$ are the standard coordinate functions on $\An$:
\begin{equation*}
x_i(a):=a_i,\quad a:=(a_1,\ldots, a_n)\in\An.
\end{equation*}
 \end{enumerate}

 In the sequel it is assumed that all considered algebraic groups are affine and all their homomorphisms are algebraic. Below all tori and diagonalizable groups are algebraic.

An action of a group $G$ on a vector space  $V$ is called {\it locally finite},
 if for every vector $v\in V$ the linear span of the orbit  $G\cdot v$ is finite dimensional.

The group $\Cr_n:={\rm Bir}\,\An$ is called  {\it the Cremona group of rank} $n$. The map $\varphi\mapsto(\varphi^*)^{-1}$ identifies it with  ${\rm Aut}_kk(x_1,\ldots, x_n)$. Every birational isomorphism  $X\dashrightarrow\An$ identifies $\Cr_n$ with ${\rm Bir}\,X$.\;For every
$g \in \Cr_n$ the functions
\begin{equation}\label{g*}
g_i=g^*(x_i)\in k(\An)
\end{equation}
determine $g$ by the formula
\begin{equation}\label{action}
g(a)=(g_1(a),\ldots, g_n(a))\;\;\mbox{if $g$ is defined at $a\in \An$};
\end{equation}
we use the notation
\begin{equation}\label{ggg}
(g_1,\ldots, g_n):=g.
\end{equation}

By means of the notion of ``algebraic family'' $S\to {\rm Cr}_n$ (see\,\cite{Ram}) the group ${\rm Cr}_n$ is endowed with the Zariski topology (see\,\cite{Serre2}, \cite{Blanc10}).
If a homomorphism $G\to\Cr_n$ of an algebraic group $G$ is
 an algebraic family, then its image is called an {\it algebraic subgroup} of $\Cr_n$ (see\,\cite{Popov3}).

 The subgroup
 \begin{equation*}
 {\Autn}:=\{(g_1,\ldots, g_n)\in \Cr_n\mid g_1,\ldots, g_n\in k[\An]=k[x_1,\ldots, x_n]\}
 \end{equation*}
 is called the {\it affine Cremona group of rank $n$}.

  It contains the algebraic subgroup of affine transformations
  \begin{equation}\label{affi}
 {\rm Aff}_n=\{(g_1,\ldots, g_n)\in \Autn \mid \deg\,g_1=\ldots=\deg\,g_n=1\},
 \end{equation}
and
  ${\rm Aff}_n$, in turn, contains the algebraic subgroup of linear transformations
   \begin{equation*}
 \GL_n=\{g\in {\rm Aff}_n \mid g(0)=0\}.
 \end{equation*}

If $g=(g_1,\ldots, g_n)\in\Autn$
(see\,\eqref{ggg}), then we put
\begin{equation*}
{\rm Jac}(g):=\det(\partial g_i/\partial x_j).
\end{equation*}
 Since $g\in\Autn$, we have
 ${\rm Jac}(g)\in k\setminus\{0\}$. Therefore, $g\mapsto {\rm Jac}(g)$ is the homomorphism of $\Autn$ in the multiplicative group of the field  $k$. Its kernel
 \begin{equation*}
 \Autsn:=\{f\in \Autn\mid {\rm Jac}(g)=1\}
 \end{equation*}
 consists of the automorphisms of  $\An$ that leaves fixed the standard volume form; it is called {\it the special affine Cremona group of rank} $n-1$ (regarding the ranks in these names see Theorems \ref{max1}(i) and \ref{max3}(i) below). The latter group contains the algebraic subgroup
 \begin{equation*}
 \SL_n:=\GL_n\cap \Autsn.
\end{equation*}

  The embeddings   $\Cr_{n}\hookrightarrow\Cr_{n+1}$, $(g_1,\ldots, g_n)\mapsto (g_1,\ldots, g_n, x_{n+1})$ are arranged the tower
  $\Cr_1\hookrightarrow\Cr_2\hookrightarrow\cdots\hookrightarrow\Cr_n\hookrightarrow\cdots$. Its direct limit $\Cr_\infty$ is called {\it the Cremona group of infinite rank},
see\,\cite[Sect.\,1]{Popov3}.

 In $\GL_n$ is distinguished the ``standard'' maximal torus
\begin{equation}\label{nDn}
D_n:=\{(t_1x_1,\ldots t_nx_n)\mid t_1,\ldots, t_n\in k\}\subset\GL_n.
\end{equation}
Its normilizer in $\GL_n$ is the group of all monomial transformations in $\GL_n$:
\begin{equation}\label{torusDn}
N_{\GL_n}(D_n)=\{(t_1x_{\sigma(1)},\ldots, t_nx_{\sigma(n)})\mid \sigma\in S_n,\; t_1,\ldots, t_n\in k\}\subset \GL_n,
\end{equation}
where $S_n$ is the symmetric group of degree $n$. In $\Autsn$ is contained the torus
\begin{equation}\label{Dnnn}
D_n^*:=D_n\cap \Autsn=\{(t_1x_1,\ldots, t_nx_n)\mid t_1,\ldots, t_n\in k,\; t_1\cdots t_n=1\}.
\end{equation}

The ``coordinate" characters $\varepsilon_1,\ldots,\varepsilon_n$ of the torus $D_n$ defined by
\begin{equation}\label{epsilonn}
\varepsilon_i\colon D_n\to \mathbf G_{\rm m}, \quad (t_1x_1,\ldots, t_nx_n)\mapsto t_i,
\end{equation}
constitute a base of the (free Abelian) group ${\rm X}(D_n)$.

\section{Some subgroups of $\Cr_n$}
In the sequel we consider the elements of the group $\mathbf Z^n$ as rows of length  $n$. Then the rows of any matrix $A=(a_{ij})\in {\rm Mat}_{m\times n}(\mathbf Z)$ become the elements of this group and we use the following notation:
\begin{gather}
{\mathcal R}_A
:=\mbox{the subgroup of ${\mathbf Z}^n$ generated by the rows of matrix $A$},\label{matrix}\\[-2pt]
D_n(A)=\bigcap_{i=1}^{m} {\rm ker}\,\lambda_i,\quad\mbox{where}\quad \lambda_i:=\varepsilon_1^{a_{i,1}}\cdots\varepsilon_n^{a_{i,n}}.\label{chi}
\end{gather}
If $m=1$, then in place of $D_n((l_1\ldots l_n))$ we write $D_n(l_1,\ldots, l_n)$. In particular,
\begin{equation}\label{all0}
D_n(0,\ldots,0)=D_n.
\end{equation}

Clearly,
$D_n(A)$ is a closed subgroup of  $D_n$ and (see\,\eqref{matrix})
\begin{equation}\label{iters}
D_n(A)=\bigcap_{(l_1,\ldots, l_n)\in{\mathcal R}(A)} {\rm ker}\,\varepsilon_1^{l_1}\cdots\varepsilon_n^{l_n}.
\end{equation}

Recall the terminology used below (see, e.g.,\;\cite{Vin} and \cite{MM}).

Every finite Abelian group $G$ decomposes as a direct sum of cyclic sub\-groups
of orders  $d_1,\ldots, d_m$, where
$d_i$ divides $d_{i+1}$ for $i=1,\ldots, m-1$, and $d_1>1$ if $|G|>1$. The numbers $d_1,\ldots, d_s$ are uniquely determined by   $G$ and are called the {\it invariant factors} of $G$.

Every nonzero integer matrix  $A$ can be transformed by means of elemen\-ta\-ry transformations of its rows and columns into a matrix $S=(s_{ij})$ such that only  the coefficients $s_{ii}$ for $i=1,\ldots, r$ are nonzero and  $s_{ii}$ divides $s_{i+1,i+1}$ for $i=1,\ldots, r-1$. The integers $s_{11},
\ldots, s_{rr}$  are uniquely determined by $A$   ($s_{ii}=f_i/f_{i-1}$, where $f_i$ is gcd of the minors of order $i$ of the matrix $A$ and $f_0:=1$) and are called  the {\it invariant factors of matrix} $A$. The matrix $S$ is called  the {\it Smith normal form of matrix} $A$.

\begin{lemma}\label{inC}  If $B$ is obtained from
$A\in {\rm Mat}_{m\times n}(\mathbf Z)$ by means of the elemen\-ta\-ry transformations of rows and columns, then the subgroups
$D_n(A)$ and $D_n(B)$ are conjugate in  $\Cr_n$.
\end{lemma}

\begin{proof} Let $\tau_1,\ldots,\tau_n$ be some base of the group  ${\rm X}(D_n)$. Then
 (see \eqref{chi}) $\lambda_i=\tau_1^{c_{i,1}}\cdots \tau_n^{c_{i,n}}$ for some  $c_{ij}\in \mathbf Z$ and
\begin{equation}\label{DAchi}
D_n(A)=\bigcap_{i=1}^m{\rm ker}\,\tau_1^{c_{i,1}}\cdots \tau_n^{c_{i,n}}.
\end{equation}

The group ${\rm Aut}_{\rm gr}D_n$ of automorphisms of the algebraic group  $D_n$ is natu\-ral\-ly identified with
  $\GL_n(\mathbf Z)$. Its natural action on the set of bases
of the group ${\rm X}(D_n)$ is transitive. Therefore, there is an automorphism
   \begin{equation}\label{phiii}
   \varphi\in
   {\rm Aut}_{\rm gr}D_n,
   \end{equation}
    such that $
   \tau_i\circ \varphi=\varepsilon_i$
    for each $i$. From \eqref{DAchi} it then follows that
    \begin{equation}\label{cphiii}
    \varphi^{-1}(D_n(A))=D_n(C),\quad\mbox{где $C=
    (c_{ij})\in{\rm Mat}_{m\times n}(\mathbf Z)$.}
     \end{equation}
     Since the map of varieties $D_n\to \An, \; (t_1x_1,\ldots, t_nx_n)\mapsto (t_1,\ldots, t_n),$ is a birational isomorphism, by means of it we can identify
the group ${\rm Cr}_n={\rm Bir}\,\An$ with the group
of birational automorphisms of
the underlying variety of torus
 $D_n$. Then $\varphi$ becomes an element of the group ${\rm Cr}_n$ and from \eqref{phiii} and \eqref{cphiii} it is easy to deduce that in this group we have the equality
     \begin{equation*}
     \varphi^{-1}D_n(A)\varphi=D_n(C).
     \end{equation*}

     Further, notice that if the base  $\tau_1,\ldots,\tau_n$ is obtained from the base $\varepsilon_1,\ldots\break\ldots, \varepsilon_n$ by an elementary transformation, then
the matrix $C$ is obtained from  $A$ by an elementary transformation of columns, and every elementary transformation of columns of $A$ is realizable in this way.
 Also, notice that if the sequence
$\varphi_1,\ldots,\varphi_m\in {\rm X}(D_n)$ is obtained by an elementary transformation from the sequence $\lambda_1,\ldots,\lambda_m$, then
$D_n(A)=\bigcap_{i=1}^m{\rm ker}\,\varphi_i$, the matrix $(c_{ij})$ defined by the equalities
$\varphi_i=\varepsilon_1^{c_{i,1}}\cdots \varepsilon_n^{c_{i,n}}$
is obtained from  $A$ by an elementary transformation of rows,
and every elementary transformation of rows of $A$ is realizable in this way.

Clearly, the said implies the claim of lemma.
\quad $\square$ \renewcommand{\qed}{}\end{proof}

\begin{corollary}\label{AS} If $S$ is the Smith normal form of matrix $A$, then
the subgroups $D_n(A)$ and $D_n(S)$ are conjugate in ${\rm Cr}_n$.
\end{corollary}

\begin{lemma}\label{sbgps}\

\begin{enumerate}[\hskip 2.2mm \rm(i)]

 \item Let $q_1\!\leqslant\!\cdots\!\leqslant\! q_r$ be all the invariant factors of matrix $A\!\in\! {\rm Mat}_{m\times n}(\mathbf Z)$.
         Then the group $D_n(A)$ is isomorphic to
        \begin{equation}\label{decompoo}
        {\boldsymbol \mu}_{q_1}
        \times \cdots\times {\boldsymbol \mu}_{q_r}
        \times {\mathbf G}_{\rm m}^{n-r}.
        \end{equation}
\item The closed $(n-m)$-dimensional subgroups of  $D_n$ are every possible subgroups  $D_n(A)$, where $A\in {\rm Mat}_{m\times n}(\mathbf Z)$, ${\rm rk}\,A=m$,
    and only they.

    \item ${\mathcal R}_A=\{(l_1,\ldots, l_n)\in {\mathbf Z}^n\mid
    D_n(A)\subseteq {\rm ker}\,\varepsilon_1^{l_1}\cdots\varepsilon_n^{l_n}\}$ for every
    $A\in {\rm Mat}_{m\times n}(\mathbf Z)$.

        \item  If $A\in {\rm Mat}_{s\times n}(\mathbf Z)$, $B\in {\rm Mat}_{t\times n}(\mathbf Z)$, then
     \begin{enumerate}[\hskip .2mm\rm(a)]

\item $D_n(A)=D_n(B)$ if and only if ${\mathcal R}_A
={\mathcal R}_B;$

\item The following properties are equivalent:
\begin{enumerate}
\item[$({\rm b}_1)$] $D_n(A)$ and $D_n(B)$ are conjugate in $\GL_n;$
 \item[$({\rm b}_2)$]  $D_n(A)$ and $D_n(B)$ are conjugate in  $N_{\GL_n}(D_n);$
 \item[$({\rm b}_2)$] by a permutation of columns, it is possible to tranform
   $B$ into a matrix  $C$ such that
${\mathcal R}_A
={\mathcal R}_C
$.
\end{enumerate}
\end{enumerate}
    \end{enumerate}
\end{lemma}

\begin{proof}
(i) Let $S=(s_{ij})$
be the normal Smith form of matrix
 $A$.\;Then
$s_{11}=q_1,\ldots, s_{rr}=q_r$ and $s_{ij}=0$ in the other cases. Hence
$D_n(S)$ is isomorphic to group \eqref{decompoo}. But $D_n(A)$ is isomorphic to $D_n(S)$ by Corollary\;\ref{AS}.\;This proves\;(i).

(ii) From (i) we deduce that
\begin{equation}\label{rk}
\dim D_n(A)=n-{\rm rk}\,A;
 \end{equation}
 therefore, $\dim D_n(A)=n-m$, where ${\rm rk}\,A=m$.\;Conversely, let $H$ be a closed subgroup of  $D_n$ such that $\dim H=n-m$. Then $D_n/H$ is an $m$-dimensional torus \cite[p.\,114]{Borel} and therefore, there is an isomorphism  $\alpha\colon D_n/H\to {\mathbf G}_{\rm m}^m$. Let $\lambda_i\in {\rm X}(D_n)$ be the composition of homomorphisms
\begin{equation*}
D_n\xrightarrow{\;\pi\;} D_n/H\xrightarrow{ \;\alpha\; } {\mathbf G}_{\rm m}^m\xrightarrow{\;{\rm pr}_i\;} {\mathbf G}_{\rm m},
\end{equation*}
where $\pi$ s the canonical projection and ${\rm pr}_i$ is the projection to the
$i$th factor. Then
$H=\bigcap_{i=1}^m {\rm ker}\,\lambda_i$. From \eqref{chi} we then deduce that  $H=D_n(A)$ and from \eqref{rk} that ${\rm rk}\,A=m$. This proves (ii).

(iii) From  \eqref{matrix}, \eqref{chi} it follows that the left-hand side of
the equality under proof is contained in the right-hand side.\;Proving the inverse inclusion
consider a character
$\lambda=\varepsilon_1^{l_1}\cdots\varepsilon_n^{l_n}
$ whose kernel contains $D_n(A)$.\;Without changing $D_n(A)$ and ${\mathcal R}_A$, we can leave in $A$ only the rows that form a base of the group ${\mathcal R}_A$ removing the other rows, i.e., we can reduce our considerations to the case where ${\rm rk}\,A=m$. Consider then the characters
 $\lambda_1,\ldots,\lambda_m$ defined by formula \eqref{chi}, and the homomorphism
\begin{equation*}
\varphi\colon D_n\to {\mathbf G}_{\rm m}^m,\quad g\mapsto (\lambda_1(g),\ldots,\lambda_m(g)).
\end{equation*}
  In view of \eqref{chi}, we have ${\rm ker}\,\varphi=D_n(A)$. From this and   \eqref{rk} it follows that $\dim\varphi(D_n)=m$. Therefore, $\varphi$ is a surjection. Hence ${\mathbf G}_{\rm m}^m$ is the quotient group of  $D_n$ by $D_n(A)$ and $\varphi$ is the canonical homomorphism to it. Since $\lambda$ is constant on the fibers of $\varphi$, the universal property of quotient implies the existence of character $\mu\colon {\mathbf G}_{\rm m}^m\to {\mathbf G}_{\rm m}$ such that $\lambda=\mu\circ\varphi$. Hence
 $\lambda=\lambda_1^{c_1}\cdots\lambda_m^{c_m}$ for some $c_1,\ldots, c_m\in{\mathbf Z}$, and this means that $(l_1,\ldots,l_n)\in {\mathcal R}_A$.
This proves (iii).

(iv)(a) If $D_n(A)=D_n(B)$, then ${\mathcal R}_A={\mathcal R}_B$ because of (iii). Conversely, if ${\mathcal R}_A={\mathcal R}_B$, then $D_n(A)=D_n(B)$ because of \eqref{iters}. This proves (iv)(a).

(iv)(b) By fusion theo\-rem \cite[1.1.1]{Serre1} the subgroups $D_n(A)$ and $D_n(B)$ are conjugate in $\GL_n$ if and only if they are conjugate in $N_{\GL_n}(D_n)$. But \eqref{torusDn} and \eqref{chi}  imply that $D_n(A)$ and $D_n(B)$ are conjugate in $N_{\GL_n}(D_n)$ if and only if by a permutation of columns one can obtain from  $B$ a matrix $C$ such that $D_n(A)=D_n(C)$. Because of (iii)(a), the latter equality is equivalent to the equality ${\mathcal R}_C={\mathcal R}_A$.
This proves (iv)(b).
\quad $\square$ \renewcommand{\qed}{}\end{proof}

\begin{corollary}
\label{cod1} \
\begin{enumerate}[\hskip 2.2mm \rm(i)]

\item Let $(l_1,\ldots, l_n)\!\neq\! (0,\ldots, 0)$  and $d\!:=\!\mbox{\rm gcd}(l_1,\ldots,l_n)$.\,Then
$D_n(l_1,\ldots,l_n)$
is isomorphic to
 ${\boldsymbol \mu}_d\times \mathbf G_{\rm m}^{n-1}$.\;In particular, the group $D_n(l_1,\ldots,l_n)$
is connected {\rm(}i.e., is a torus{\rm)} if and only if $d=1$.

\item The closed $(n-1)$-dimensional subgroups of  $D_n$ are every possible subgroups $D_n(l_1,\ldots,l_n)$ with  $(l_1,\ldots, l_n)\neq (0,\ldots, 0)$ and only they.

    \item $D_n(l_1,\ldots,l_n)=D_n(l'_1,\ldots,l'_n)$ if and only if
    \begin{equation*}
    (l_1,\ldots,l_n)=\pm(l'_1,\ldots,l'_n).
    \end{equation*}

    \item The following properties are equivalent:
\begin{enumerate}
\item[$({\rm iv}_1)$]
$D_n(l_1,\ldots,l_n)$ and $D_n(l'_1,\ldots,l'_n)$ are conjugate in $\GL_n;$
\item[$({\rm iv}_2)$] $D_n(l_1,\ldots,l_n)$ and $D_n(l'_1,\ldots,l'_n)$ are conjugate in $N_{\GL_n}(D_n);$
\item[$({\rm iv}_3)$] there is a permutation
 $\sigma\in S_n$ such that
\begin{equation*}\label{sigm}
(l_1,\ldots, l_n)=\pm(l'_{\sigma(1)},\ldots, l'_{\sigma(n)}).
\end{equation*}
\end{enumerate}
\end{enumerate}
\end{corollary}

The following Lemma \ref{equi} gives an effective numerical criterion for the equa\-li\-ty  ${\mathcal R}_A={\mathcal R}_B$ from Lemma \ref{sbgps}(iii).

 Let $A\in {\rm Mat}_{m\times n}({\mathbf Z})$, ${\rm rk}\,A=m$. For every strictly
increasing sequence of  $m$
 integers $i_1,\ldots, i_m$ taken from the interval $[1, n]$, put
\begin{equation}\label{pluck}
p_{i_1,\ldots, i_m}(A):={\rm det}\,A_{i_1,\ldots, i_m},
\end{equation}
where $A_{i_1,\ldots, i_m}$ is the submatrix of matrix $A$ obtained by intersecting rows with numbers $1,\ldots, m$ and columns with numbers $i_1,\ldots, i_m$ (it is natural to call the $p_{i_1,\ldots, i_m}(A)$'s the {\it Pl\"ucker coordinates of matrix} $A$).

\begin{lemma}\label{equi} For every two matrices $A$  and  $B\in {\rm Mat}_{m\times n}(\mathbf Z)$  of rank $m$
the following properties are equivalent$:$
\begin{enumerate}[\hskip 2.2mm \rm(i)]
\item ${\mathcal R}_A
={\mathcal R}_B.$
\item Two conditions hold:
\begin{enumerate}[\hskip .2mm \rm(a)]
\item either $p_{i_1,\ldots, i_m}(A)=p_{i_1,\ldots, i_m}(B)$ for all $i_1,\ldots, i_m$, or $p_{i_1,\ldots, i_m}(A)=-p_{i_1,\ldots, i_m}(B)$ for all $i_1,\ldots, i_m;$
    \item   for every sequence $i_1,\ldots, i_m$ such that $p_{i_1,\ldots, i_m}(A)\neq 0$, the fol\-low\-ing inclusion holds:
     \begin{equation}\label{determ}
     B_{i_1,\ldots, i_m}(A_{i_1,\ldots, i_m})^{-1}\in {\rm Mat}_{m\times m}(\mathbf Z).
     \end{equation}
\end{enumerate}
\end{enumerate}
\end{lemma}

\begin{proof}  Whereas ${\rm rk}\,A={\rm rk}\,B=m$, the rows of $A$ and $B$ form the bases in  ${\mathcal R}_A$ and ${\mathcal R}_B$ respectively. Therefore, ${\mathcal R}_A={\mathcal R}_B$ if and only if when there is a matrix  $Q\in\GL_m({\mathbf Z})$ such that
\begin{equation}\label{C}
A=QB.
\end{equation}

Let ${\mathcal R}_A={\mathcal R}_B$. Then \eqref{C} implies that
\begin{equation}\label{ABC}
A_{i_1,\ldots, i_m}=QB_{i_1,\ldots, i_m}
 \end{equation}
 for every $i_1,\ldots, i_m$, and therefore, $p_{i_1,\ldots, i_m}(A)=\det Q p_{i_1,\ldots, i_m}(B)$ because of \eqref{pluck}.
Since $Q\in\GL_m({\mathbf Z})$, we have $\det Q=\pm 1$. Hence condition (ii)(a) holds.
If $p_{i_1,\ldots, i_m}(A)\neq 0$, then $p_{i_1,\ldots, i_m}(B)\neq 0$ as well, hence $B_{i_1,\ldots, i_m}$ is nondegenerate and \eqref{ABC} implies that $Q=A_{i_1,\ldots, i_m}(B_{i_1,\ldots, i_m})^{-1}$.\;Hence condition (ii)(b) holds. This proves (i)$\Rightarrow$(ii).

Proving the inverse implication, consider
${\mathbf Z}^n$ as a subset of the coordinate vectors space (of rows)  ${\mathbf Q}^n$. Condition (ii)(a) shows that  the ${\mathbf Q}$-linear spans of subsets  ${\mathcal R}_A$ and ${\mathcal R}_B$ of ${\mathbf Q}^n$ have the same Pl\"ucker coordinates. Hence these spans are one and the same linear subspace
 $L$ (see, e.g.,\;\cite[Theorem 10.1]{SR}). Since the rows of $A$ and rows of $B$ form two bases in  $L$, there exists a matrix $P\in \GL_m(\mathbf Q)$ such that $A=PB$. Therefore, $A_{i_1,\ldots, i_m}=PB_{i_1,\ldots, i_m}$ for every
$i_1,\ldots, i_m$ and hence $P=A_{i_1,\ldots, i_m}(B_{i_1,\ldots, i_m})^{-1}$ if
$p_{i_1,\ldots, i_m}(B)\neq 0$. Then (ii)(b) implies that $P\in \GL_n(\mathbf Z)$. Therefore, ${\mathcal R}_A={\mathcal R}_B$. This proves (ii)$\Rightarrow$(i).
\quad $\square$ \renewcommand{\qed}{}\end{proof}

\begin{remark} \

$1$. {\rm In the proof of Lemma \ref{equi} it is established that, in fact,
 (ii)(a) implies that the matrix $B_{i_1,\ldots, i_m}(A_{i_1,\ldots, i_m})^{-1}$ is independent of the choice of a sequence  $i_1,\ldots, i_m$ for which $p_{i_1,\ldots, i_m}(A)\neq 0$. Therefore,  (ii)(b) follows from (ii)(a) and feasibility of \eqref{determ} for any {\it one} such a sequence.}

$2$. {\rm If $m\!=\!1$, then
 (ii)(b) follows from (ii)(a) (but for $m>1$ this is not so).}
\end{remark}

\begin{theorem}[Classificaton of diagonalizable subgroups of $\Aff_n$ up to conju\-ga\-cy in ${\rm Cr}_n$]\label{dc}\

 \begin{enumerate}[\hskip 2.2mm \rm(i)]
 \item Two diagonalizable subgroups of the group $\Aff_n$ are conjugate in $\Cr_n$ if and only if they are isomorphic.
 \item
  Any diagonalizable subgroup $G$ of the group $\Aff_n$ is conjugate in $\Cr_n$ to a unique closed subgroup of the torus $D_n$ that has the form
  \begin{equation}\label{canon}
   {\rm ker}\,\varepsilon_{r+1}^{d_1}\cap\ldots \cap {\rm ker}\,\varepsilon_{r+s}^{d_s}\cap {\rm ker}\,\varepsilon_{r+s+1}\cap\ldots \cap {\rm ker}\,\varepsilon_{n}
\end{equation}
where $0\leqslant r\leqslant n$, $0\leqslant s\leqslant n$, $r+s\leqslant n$,
$2\leqslant d_1$ and  $d_i$ divides $d_{i+1}$ for every $i<s$. The integers determining subgroup {\rm\eqref{canon}} have the following meaning: $r=\dim G$ and
$d_1,\ldots, d_s$ are all the invariant factors of the finite Abelian group $G/G^0$.
 \end{enumerate}
\end{theorem}

\begin{proof} Since
maximal reductive subgroups of algebraic group are conjugate
 (see\,\cite[5.1]{BS}) and $\GL_n$ is one of them in  $\Aff_n$, every diagonalizable subgroup of  $\Aff_n$ is conjugate to a subgroup of $\GL_n$. In turn, every diago\-na\-lizable subgroup of $\GL_n$ is conjugate to a subgroup of the torus $D_n$ (see\,\cite[I.4.6]{Borel}). In view of Lemma \ref{sbgps}(ii) this shows that it suffices to prove (i) for the subgroups $D_n(A)$ and $D_n(B)$ of the torus $D_n$. Further, appending, if necessary, zero rows we may assume that
 $A$ and $B$ have the same number of rows. Let now $D_n(A)$ and $D_n(B)$ be isomorphic. Then their dimensions are equal and the groups $D_n(A)/D_n(A)^0$ and
$D_n(B)/D_n(B)^0$ have the same invariant factors. This and Lemma \ref{sbgps}(i) imply
that the matrices $A$ and $B$ have the same invariant factors (because the latter are obtained by appending the same number of 1's to the invariant factors of the specified  groups). Hence the normal Smith forms of $A$ and $B$ are equal. By Corollary \ref{AS} this implies that  $D_n(A)$ and $D_n(B)$ are conjugate in ${\Cr}_n$. This proves (i).

It is clear that the integers determining subgroup {\rm\eqref{canon}} have the meaning specified in (ii). Since every diagonalizable group is a direct product of a finite Abelian group and a torus, it is uniquely, up to isomorphism, determined by its dimension and the invariant factors of the group of connected components. This and (i) implies (ii).
\quad $\square$ \renewcommand{\qed}{}\end{proof}

\begin{corollary} $D_n(A)$ and $D_n(B)$ are conjugate in ${\rm Cr}_n$ if and only if  $A$ and $B$ have the same invariant factors.
\end{corollary}

\begin{corollary}
\label{DCr}
Every torus $T$ in $\Aff_n$ is conjugate in ${\rm Cr}_n$ to the torus $D_r$, где $r=\dim T$.
\end{corollary}

\begin{proof} This follows from Theorem \ref{dc}(ii).
\quad $\square$ \renewcommand{\qed}{}\end{proof}

\begin{corollary}\label{isoconj} Every two isomorphic finite Abelian subgroups of  $\Aff_n$ are conjugate in $\Cr_n$.
\end{corollary}

\begin{proof} Since ${\rm char}\,k=0$, every element in $\Aff_n$ of finite order is semisimple. Hence every finite Abelian subgroup of $\Aff_n$ is reductive, and therefore, is conjugate in $\Aff_n$ to a subgroups of  $\GL_n$ (see the proof of Theorem \ref{dc}).
But every commutative subgroup of  $\GL_n$ consisting only of semisimple elements is diagonalizable (see\,\cite[Prop.\,4.6(b)]{Borel}).
The claim now follows from Theorem \ref{dc}(i).
\quad $\square$ \renewcommand{\qed}{}\end{proof}

\begin{corollary}[{\cite[Thm.\,1]{Blanc06}}] Every two elements in
 $\Aff_n$ of the same finite order are conjugate in  ${\rm Cr}_n$.
\end{corollary}

\section{Tori in ${\rm Cr}_n$, $\Autn$, and ${\rm Aut}^*{\mathbf A}^n$}

\begin{theorem}[Tori in ${\rm Cr}_n$]\label{max1} \

\begin{enumerate}[\hskip 2.2mm\rm(i)]
\item In ${\rm Cr}_n$ there are no tori of dimension $>n$.
\item
In ${\rm Cr}_n$ every $r$-dimensional torus for  $r=n, n-1, n-2$
is conjugate to the torus $D_r$.
    \item If $n\geqslant 5$, then in ${\rm Cr}_n$ there are $(n-3)$-dimensional tori that are not conjugate to subtori of the torus
       $D_n$.
\item Every $r$-dimensional torus in\,${\rm Cr}_n$ is conjugate in\,${\rm Cr}_{n+r}$
to the torus\,$D_r\!.$
  \item  In ${\rm Cr}_{\infty}$ every $r$-dimensional torus is conjugate to the torus $D_r$.
\end{enumerate}
\end{theorem}

\begin{proof}
(i) This is proved, e.g., in \cite{Dem}, see also \cite[Cor.\,2.2]{Popov3}.

(ii) According to \cite[Cor.\,2]{BB1} (see also \cite[Cor.\,2.4(b)]{Popov3}), every $r$-dimensional torus in ${\rm Cr}_n$ for $r=n, n-1, n-2$ is con\-ju\-ga\-te to a subtorus of the torus $D_n$. Therefore, the claim follows from Corollary \ref{DCr}.

(iii) This is proved in \cite[Cor.\,2.5]{Popov3}.

(iv) According to \cite[Thm.\,2.6]{Popov3}, every $r$-dimensional torus in  ${\rm Cr}_r$
is conju\-ga\-te in ${\rm Cr}_{n+r}$ to a subtorus of the torus $D_{n+r}$. Therefore, the claim follows from Corollary \ref{DCr}.

(v) This follows from (iv).
\quad $\square$ \renewcommand{\qed}{}
\end{proof}

\begin{corollary}\label{c9} \

\begin{enumerate}[\hskip 2.2mm\rm(i)]
\item In ${\rm Cr}_n$ every $n$-dimensional torus is maximal.
\item  In ${\rm Cr}_n$ there are no maximal $(n-1)$- and $(n-2)$-dimensional tori.
\item For $n\geqslant 5$, in ${\rm Cr}_n$ there are maximal  $(n-3)$-dimensional tori.
    \end{enumerate}
 \end{corollary}

 \begin{remark} {\rm For $n\leqslant 3$, Theorem \ref{max1} gives the classification of tori in  ${\rm Cr}_n$ up to conjugacy: the classes of conjugate nontrivial tori are exhausted by that of the tori  $D_1,\ldots, D_n$.}
 \end{remark}

 Let $(l_1,\ldots, l_n)\in \mathbf Z^n$ be a nonzero element with  gcd$(l_1, \ldots, l_n)$=1. Clearly, the homomorphism
\begin{equation}\label{embd}
{\mathbf G}_{\rm m}\to D_n, \quad t\mapsto (t^{l_1}x_1,\ldots, t^{l_n}x_n),
 \end{equation}
is an embedding and every embedding  ${\mathbf G}_{\rm m}\hookrightarrow D_n$ is of this form. Denote by $T(l_1,\ldots, l_n)$ the image of embedding \eqref{embd}. It is a one-dimensional torus in $D_n$ and every one-dimensional torus in $D_n$ is of this form.

 \begin{lemma}\label{TTTT} The following properties are equivalent:
 \begin{enumerate}[\hskip 2.2mm\rm(i)]
\item $T(l_1,\ldots, l_n)=T(l_1',\ldots, l_n')$.
\item $(l_1,\ldots, l_n)=\pm(l_1',\ldots, l_n')$.
\end{enumerate}
 \end{lemma}

 \begin{proof} (i)$\Rightarrow$(ii) This is clear.

(ii)$\Rightarrow$(i) Assume that (i) holds. Take an element $t\in {\mathbf G}_{\rm m}$ of infinite order. From (i) and the definition of  $T(l_1,\ldots, l_n)$ we deduce that there is an element $s\in {\mathbf G}_{\rm m}$ such that
$t^{l_i}=s^{l_i'}$
for every $i=1,\ldots, n$. Hence $t^{l_il_j'}=s^{l_i'l_j'}=t^{l_jl_i'}$
for every two nonequal integers $i$ and $j$ taken from the interval $[1, n]$. Since the order of $t$ is infinite, this implies that $l_il_j'-l_jl_i'=0$. Hence
 $$
 {\rm rk}\begin{pmatrix}
 l_1 &\ldots & l_n\\
  l_1' &\ldots & l_n'
 \end{pmatrix}=1.
 $$
Therefore, $(l_1,\ldots, l_n)=\gamma  (l_1',\ldots, l_n')$ for some $\gamma\in\mathbf Q$, or, equivalently, $p(l_1,\ldots, l_n)=q(l_1',\ldots, l_n')$, where  $p, q\in\mathbf Z$, gcd$(p, q)=1$. Hence $p$ divides each of $l_1',\ldots,l_n'$, and $q$ divides each of $l_1,\ldots,l_n$. Since  the integers
in each of these two sets are coprime, this implies that
 $\gamma=\pm 1$, i.e.,\;(ii) holds.
\quad $\square$ \renewcommand{\qed}{}
\end{proof}

\begin{theorem}[Tori in $\Autn$]\label{max2}\

\begin{enumerate}[\hskip 2.2mm\rm(i)]
\item
In $\Autn$
every $n$-dimensional torus is conjugate to the torus
$D_n$.
\item In $\Autn$ all $(n-1)$-dimensional tori are exhausted, up to conjugacy in $\Autn$, by
all the groups of the form
 $D_n(l_1,\ldots, l_n)$ with
$(l_1,\ldots, l_n)\neq (0,\ldots,0)$ and {\rm gcd}$(l_1,\ldots, l_n)=1$.

\item In ${\rm Aut}\,{\mathbf A}^3$ all one-dimensional tori are exhaused, up to conjugacy, by all the groups of the form $T(l_1, l_2, l_3)$.
\end{enumerate}
\end{theorem}

\begin{proof}  According to \cite{BB1} and \cite{BB2}, every $r$-dimensional torus in the group $\Autn$ for $r=n$ and $n-1$, respectively,  is conjugate to a subtorus of the torus $D_n$. This and Corollary \ref{cod1} imply (i) and (ii).

According to \cite{Aut3}, every one-dimensional torus in ${\rm Aut}\,{\mathbf A}^3$ is conjugate to subtorus of the torus $D_3$. This implies (iii).
\quad $\square$ \renewcommand{\qed}{}
\end{proof}

\begin{theorem}[Tori in $\Autsn$]\label{max3} \

\begin{enumerate}[\hskip 2.2mm\rm(i)]
\item In $\Autsn$ there are no tori of dimension $>n-1$.
\item In  $\Autsn$ every $(n-1)$-dimensional torus is maximal and conjugate to the torus  $D_n^*$ {\rm(}see\,\eqref{Dnnn}{\rm)}.
\end{enumerate}
\end{theorem}

\begin{proof} (i) If $\Autsn$ contains an $n$-dimensional torus $T$, then by Theorem \ref{max2}(i) there is an element $g\in \Autn$ such that
\begin{equation}\label{TD}
T=gD_ng^{-1}.
 \end{equation}
Replacing $g$ by $gz$, where $z\in D_n$ is an element such that ${\rm Jac}(z)=\det z=1/{\rm Jac}(g)$, we may assume that
$g\in \Autsn$. This and \eqref{TD} imply that $D_n\subset \Autsn$,\,---\,a contradiction. This proves (i).

(ii) Let $S$ be an $(n-1)$-dimensional torus in $\Autsn$. By Theorem \ref{max2}(i) and Corollary (ii) of Lemma \ref{sbgps}, there are $g\in \Autn$ and $(l_1,\ldots, l_n)\in {\mathbf Z}^n$ such that
\begin{equation}\label{sTD}
S=gD_n(l_1,\ldots, l_n)g^{-1}.
\end{equation}
As in the proof of (i), we may assume that $g\in \Autsn$. From \eqref{sTD} it then follows that $D_n(l_1,\ldots, l_n)\subset \Autsn$. This and $D_n\cap {\rm ker}\,{\rm Jac}=D_n^*:=D_n(1,\ldots,1)$ imply the inclusion $D_n(l_1,\ldots, l_n)\subseteq D_n^*$. By Corollary (i) of Lemma \ref{sbgps}, both sides of this inclusion are $(n-1)$-dimensional tori. Hence it is the equality. This proves (ii)
\quad $\square$ \renewcommand{\qed}{}
\end{proof}

\begin{remark}\label{z} {\rm   For  $\Autn$ and $\Autsn$, unlike for ${\rm Cr}_n$ (see Corollary \ref{c9}(iii)), at present nothing is known on the existence of tori of nonmaximal dimension. This problem is intimately related to the
 {\it Cancellation Problem}: Is there an affine variety $X$ not isomorphic to ${\mathbf A}^m$, $m=\dim X$, such that $X\times {\mathbf A}^d$ is isomorphic to ${\mathbf A}^{m+d}$ for some $d$?
If the answer is positive, then in $\Autn$ and $\Autsn$ with $n=m+d$ there exists a maximal torus  $T$
of a nonmaximal dimension. Indeed, multiplying, if necessary, $X\times {\mathbf A}^d$ by ${\mathbf A}^1$, one may assume that  $d\geqslant 2$. Let $\lambda$ be the character $t\mapsto t$ of the torus ${\mathbf G}_{\rm m}$. Consider the linear action of  ${\mathbf G}_{\rm m}$ on ${\mathbf A}^d$
with precisely two isotypic components:
one is $(d-1)$-dimensional of type ${\lambda}$, another is one-dimensional of typer ${\lambda}^{1-d}$.  It determines the action of ${\mathbf G}_{\rm m}$ on
$X\times {\mathbf A}^d$ via the second  factor and hence, an action of ${\mathbf G}_{\rm m}$ on $\An$. Consider in $\Autn$ the torus that is the image of ${\mathbf G}_{\rm m}$ under the homomorphism determined by this action. The construction implies that this torus lies in  $\Autsn$. Let $T$ (respectively, $T'$) be a maximal torus of  $\Autn$ (respectively, of  $\Autsn$) that contains this image.\;If $T$ is $n$-dimensional (respectively, $T'$ is $(n-1)$-dimensional), then by Theorem \ref{max2}(i) (respectively, by Theorem \ref{max3}) it is conjugate to a subtorus of the torus $D_n$ and, therefore, the specified action of  $T$ (respectively, $T'$) on $\An$ is equivalent  to a linear one. Hence the action of ${\mathbf G}_{\rm m}$ on $\An$ is equivalent to a linear one as well, and therefore, the set $F$ of its fixed points is isomorphic to an affine space. But, by the construction, in ${\mathbf A}^d$ there is a single ${\mathbf G}_{\rm m}$-fixed point, consequently, $F$ is isomorphic to $X$\,---\,a contradiction.\footnote{{\it Added on August} 2, 2012. Theorems \ref{max2} and \ref{max3} and Remark \ref{z} remain valid
when the characteristic
 $p$ of the ground field $k$ is positive.\;In the preprint N. Gupta, {\it A counter-example to the Cancellation Problem for the affine space ${\mathbf A}^3$ in characteristic $p$}, {\tt arXiv:1208.0483} published today, is proved that the hypersurface
$X$ in ${\mathbf A}^4$ defined by the equation
 $x_1^mx_2+$ $x_3^{p^e}+x_4+x^{sp}=0$, where
$m, e, s$ are positive integers such that $p^e{\hskip -.3mm}\nmid{\hskip -.3mm}sp$ and $sp{\hskip -.3mm}\nmid{\hskip -.3mm}p^e$, is not isomorphic to ${\mathbf A}^3$, but $X\times {\mathbf A}^1$ is isomorphic to ${\mathbf A}^4$.\;By Remark \ref{z},
this implies that, for ${\rm char}\,k>0$ and every integer $n\geqslant 5$, there are
 maximal tori of nonmaximal dimension in $\Autn$ and $\Autsn$.}
}
 \end{remark}

\section{Orbits and stabilizers of the action of $D_n(l_1,\ldots, l_n)$ on $\An$ }

Here we establish some properties, needed for what follows, of orbits and stabilizers of the natural action on $\An$ of the group
\begin{equation}\label{G}
G:=D_n(l_1,\ldots, l_n).
\end{equation}

Clearly, every coordinate hyperplane
\begin{gather}
H_i:=\{a\in \An\mid x_i(a)=0\}
\end{gather}
is $G$-invariant.

\begin{lemma}\label{op} The $G$-stabilizer of every point of  $\An\setminus \bigcup_{i=1}^n H_i$ is trivial.
\begin{enumerate}[\hskip 2.2mm\rm(i)]
\item If $(l_1,\ldots, l_n)=(0,\ldots, 0)$, then $\An\setminus \bigcup_{i=1}^n H_i$ is a  $G$-orbit, and dimen\-sion of the $G$-stabilizer of every point of  $\;\bigcup_{i=1}^n H_i$ is positive.
\item If $(l_1,\ldots, l_n)\neq (0,\ldots, 0)$, then $\dim G\cdot a=n-1$ for every point $a\in\An\setminus \bigcup_{i=1}^n H_i$.
\end{enumerate}
\end{lemma}

\begin{proof} This immediately follows from  \eqref{G}, \eqref{chi}, \eqref{action} and Lem\-ma\;\ref{cod1}.
\quad $\square$ \renewcommand{\qed}{}\end{proof}

We now consider the case
\begin{equation*}
(l_1,\ldots, l_n)\neq (0,\ldots, 0).
\end{equation*}

\begin{lemma}\label{i0}
If $l_i\neq 0$, then the open subsubset $\mathcal O_i:=H_i\setminus \bigcup_{j\neq i} H_j$ in
 $H_i$ is a
  $G$-orbit.
\end{lemma}
\begin{proof} By virtue of $G$-invariance of $\mathcal O_i$,
it suffices to show that $\mathcal O_i$ is contained in a $G$-orbit. Since $l_i\neq 0$, the equation $x^{l_i}=\alpha$ has a solution for any  $\alpha\in k$, $\alpha\neq 0$. This and \eqref{G}, \eqref{chi} imply that for any two points
\begin{equation*}
b\!=\!(b_1,\ldots,b_{i-1}, 0, b_{i+1},\ldots, b_n),\; c\!=\!(c_1,\ldots,c_{i-1}, 0, c_{i+1},\ldots, c_n)\!\in\! H_i\setminus 
\bigcup_{j\neq i} H_j
\end{equation*}
there exists an element  $g=(t_1x_1,\ldots, t_nx_n)\in G$ such that
$t_j=b_j^{-1}c_j$ for every $j\neq i$. From \eqref{action} we then obtain that
$g\cdot b=c$,  as required.
\quad $\square$ \renewcommand{\qed}{}\end{proof}

\begin{lemma}\label{stable} The following properties are equivalent:
\begin{enumerate}[\hskip 2.2mm\rm(i)]
\item all the numbers $l_1,\ldots, l_n$ are nonzero and have the same sign;
\item $G$-orbit of every point of $\An\setminus \bigcup_{i=1}^n H_i$ is closed in $\An$;
\item $G$-orbit of some point of  $\An\setminus \bigcup_{i=1}^n H_i$ is closed in $\An$.
 \end{enumerate}
\end{lemma}
\begin{proof} Consider a point
\begin{equation*}
a=(a_1,\ldots, a_n)\in\An\setminus \textstyle\bigcup_{i=1}^n H_i.
 \end{equation*}

Assume that (i) holds. Suppose that the $G$-orbit of the point $a$
is not closed in $\An$. Then its boundary can be accessed by a one-parameter subgroup, i.e.,
there is homomorphism
  $\varphi\colon \mathbf G_{\rm m}\to G$ such that there is a limit
\begin{equation}\label{chb}
{\rm lim}_{t\to 0} \varphi(t)\cdot  a=b\in \overline{G\cdot a}\setminus G\cdot a
\end{equation}
(see\;\cite[Thm.\;6.9]{PV}); the latter means that the morphism
$\mathbf G_{\rm m}={\bf A}^1\setminus \{0\}\to \An$,\; $t\mapsto \varphi(t)\cdot  a$,
extends to a morphism ${\bf A}^1\to\An$ that maps $0$ to the point $b$. As $\varphi$ is algebraic,  there is a vector
$(d_1,\ldots, d_n)\in \mathbf Z^n$ such that $\varphi(t)=(t^{d_1}x_1,\ldots, t^{d_n}x_n)$ for every  $t\in \mathbf G_{\rm m}$. Since $\varphi(t)\cdot  a=(t^{d_1}a_1,\ldots, t^{d_n}a_n)$ and $a_i\neq 0$ for every $i$, the existence of the specified limit means that
\begin{equation}\label{posi}
d_1\geqslant 0,\ldots, d_n\geqslant 0.
\end{equation}
On the other hand, it follows from $\varphi(t)\in G$ and  \eqref{G}, \eqref{chi} that
$t^{d_1l_1+\cdots+d_nl_n}\break =1$ for every $t$, i.e.,
\begin{equation}\label{=0}
d_1l_1+\cdots+d_nl_n=0.
\end{equation}
But \eqref{posi}, \eqref{=0} and condition (i) imply that $d_1=\cdots=d_n=0$. Hence $b=a$ contrary to \eqref{chb}. This contradiction proves (i)$\Rightarrow$(ii).

Conversely, assume that (i) is not fulfilled, i.e., among $l_1,\ldots, l_n$ there are either two nonzero numbers with different signs, or one number equal to zero. In the first case, let, for instance, be
 $l_1>0, l_2<0$. Then \eqref{chi} implies that the image of the homomorphism
$\varphi\colon \mathbf G_{\rm m}\to D_n$, $t\mapsto (t^{-l_2}x_1, t^{l_1}x_2, x_3,\ldots, x_n)$, lies in $G$.  Since
\begin{equation*}
{\rm lim}_{t\to 0} \varphi(t)\cdot  a=(0, 0, a_3,\ldots, a_n)\notin G\cdot  a,
\end{equation*}
this shows that the orbit  $G\cdot  a$ is not closed. In the second case, let, say, $a_1=0$.
Then $G$ contains the image of the homomorphism $\varphi\colon\mathbf G_{\rm m}\to D_n$, $t\mapsto (tx_1, x_2,\ldots, x_n)$ and, since
\begin{equation*}
{\rm lim}_{t\to 0} \varphi(t)\cdot  a=(0, a_2,\ldots, a_n)\notin G\cdot  a,
\end{equation*}
the orbit $G\cdot a$ is not closed. This proves (iii)$\Rightarrow$(i).
\quad $\square$ \renewcommand{\qed}{}\end{proof}

From Lemmas \ref{op}, \ref{i0}, and \ref{stable} we infer

\begin{corollary} If all the numbers
$l_1,\ldots, l_n$ are nonzero and have the same sign, then in $\An$ there are precisely  $n$ nonclosed $(n-1)$-dimensional $G$-orbits\,---\,these are the orbits $\mathcal O_1,\ldots,\mathcal O_n$ from Lemma {\rm \ref{i0}}.
\end{corollary}

\begin{remark}
{\rm Recall from \cite{Pst} that an action of an algebraic group on an algebraic variety is called
{\it stable}, if orbits of points in general position are closed. Lemma \ref{stable} shows that the following properties are equivalent:}
\begin{enumerate}[\hskip 2.2mm\rm(i)]
\item {\rm all the numbers  $l_1,\ldots, l_n$ are nonzero and have the same sign;}
\item {\rm the action  $G$ on $\An$ is stable.}
\end{enumerate}
\end{remark}

\begin{lemma}\label{npm}
Assume that none of the numbers
$l_1,\ldots, l_n$
are equal to
$\pm 1$. Then the following properties of the point  $a=(a_1,\ldots, a_n)\in\An$
are equivalent:
\begin{enumerate}[\hskip 2.2mm\rm(i)]
\item $a$ has a nontrivial $G$-stabilizer;
\item $a\in\bigcup_{i=1}^n H_i$.
\end{enumerate}
\end{lemma}
\begin{proof} Lemma \ref{op} implies (i)$\Rightarrow$(ii). Now assume that (ii) holds.
Then there are the indices  $i_1,\ldots, i_s$, where $s\geqslant 1$, such that $a_j=0$ for $j=i_1,\ldots, i_s$ and $a_j\neq 0$ for other $j$'s. It follows from \eqref{action}, \eqref{G}, and \eqref{chi} that an element  $(t_1x_1,\ldots, t_nx_n)\in D_n$ lies in the  $G$-stabilizer of the point  $a$ if and only if $t_j=1$ for $j\neq i_1,\ldots, i_s$ and
\begin{equation}\label{eq}
t_{i_1}^{l_{i_1}}\cdots t_{i_s}^{l_{i_s}}=1.
\end{equation}
 Since there is no $\pm 1$ among the numbers $l_1,\ldots, l_n$ and $k$ is an algebraically closed field of characteristic zero, equality \eqref{eq}, considered as the equation in $t_{i_1},\ldots, t_{i_s}$, has at least two solutions. Hence the $G$-stabilizer of the point $a$ is nontrivial. This proves (ii)$\Rightarrow$(i).
\quad $\square$ \renewcommand{\qed}{}\end{proof}

\begin{lemma}\label{clozero}
If among $l_1,\ldots, l_n$ there are two nonzero numbers with different signs, then every
 $G$-orbit contains $(0,\ldots,0)$ in its closure.
\end{lemma}
\begin{proof} For definiteness, let, for instance, be
\begin{equation*}
l_1>0,\; l_2\geqslant 0,\ldots, l_s\geqslant 0,\; l_{s+1}<0,\; l_{s+2}\leqslant 0,\ldots, l_n\leqslant 0.
\end{equation*}
Choose a positive integer $d$ so large that
\begin{equation*}
q:=l_2+\cdots+l_s+dl_{s+1}+l_{s+2}+\cdots+l_n<0.
\end{equation*}
 Since $-ql_1+l_1l_2+\cdots + l_1l_s+l_1dl_{s+1}+l_1l_{s+2}+\ldots+l_1l_n=0$, from \eqref{G} and \eqref{chi} we infer that the image of the homomorphism
\begin{equation}\label{q}
\varphi\colon \mathbf G_{\rm m}\to D_n,\;\;
t\mapsto (t^{-q}x_1, t^{l_1}x_2,\ldots, t^{l_1}x_{s}, t^{l_{1}d}x_{s+1},
t^{l_1}x_{s+2}\ldots, t^{l_1}x_n)
\end{equation}
lies in $G$. On the other hand, since the numbers $-q, l_1$ and $d$ are positive, \eqref{q}
implies that for every point  $a\in\An$ the limit ${\rm lim}_{t\to 0}\varphi(t)\cdot  a$ exists and is equal to $(0,\ldots, 0)$.
\quad $\square$ \renewcommand{\qed}{}\end{proof}

Now consider the case where $0$ and $\pm1$ are contained among
 $l_1,\ldots, l_n$, there are at least two nonzero $l_i$'s, and all of them have the same sign. By \eqref{chi}, without loss of generality we may assume that this sign is positive.
 Up to replacing  the group $G$ by its conjugate by means of an element of $N_{\GL_n}(D_n)$, we may then assume that
\begin{equation}\label{lxxx}
\begin{split}
l_1=\ldots=l_p&=1,\; l_{p+1}\geqslant 2,\ldots, l_{q}\geqslant 2, \;l_{q+1}=\ldots=l_n=0,\\
&\mbox{ where\; $p\geqslant 1$,\;$n>q\geqslant p$\;\;and\;\;$q\geqslant 2$}.
\end{split}
\end{equation}

\begin{lemma}\label{10} Assume that \eqref{lxxx} holds. Take a point $a=(a_1,\ldots,a_n)\in\An$.
\begin{enumerate}[\hskip 2.2mm\rm(i)]
\item Let $a\notin
\bigcup_{i=1}^n H_i$. Then the orbit $G\cdot b$, where
$b=(a_1,\ldots, a_q,0,\ldots,0)$, lies in the closure of the orbit
$G\cdot  a$,
is closed, and  $\dim G\cdot  b=q-1$.
\item Let $a\in H_i$. Then the group $G_a$ is
\begin{enumerate}[\hskip .9mm\rm(a)]
\item  trivial if $1\leqslant i\leqslant p$ and $a\in \mathcal O_i$ {\rm(}see Lemma {\rm\ref{i0});}
\item  nontrivial and finite if  $p+1\leqslant i\leqslant q$ and $a\in \mathcal O_i;$
\item  has positive dimension if  $i>q$.
\end{enumerate}
\end{enumerate}
\end{lemma}
\begin{proof}    From \eqref{G}, \eqref{chi}, and \eqref{lxxx} it follows that the image of the homomorphism
\begin{equation*}\label{q}
\varphi\colon \mathbf G_{\rm m}\to D_n,\quad
t\mapsto (x_1,\ldots, x_{q}, tx_{q+1},
\ldots, tx_n),
\end{equation*}
lies in $G$. Hence the point
\begin{equation}\label{b}
{\rm lim}_{t\to 0}\varphi(t)\cdot a=(a_1,\ldots, a_q,0,\ldots,0)=b
\end{equation}
lies in the closure of the orbit  $G\cdot  a$.

Now let $a\notin \bigcup_{i=1}^n H_i$. If the orbit $G\cdot  b$ is not closed, then, as in the proof of Lemma \ref{stable}, there exist a homomorphism
\begin{equation}\label{psi}
\psi\colon \mathbf G_{\rm m}\to G, \quad t\mapsto (t^{d_1}x_1,\ldots, t^{d_n}x_n),
\end{equation}
such that $c:={\rm lim}_{t\to 0}\psi(t)\cdot b\in \overline{G\cdot b}\setminus G\cdot b$.
From \eqref{b} it follows that $d_1\geqslant 0,\ldots, d_q\geqslant 0$, and from \eqref{lxxx}, \eqref{G}, and \eqref{chi} it follows that $d_1l_1+\cdots+d_ql_q=0$. Since $l_1,\ldots,l_q$ are positive,
this yields $d_1=\ldots=d_q=0$. In view of \eqref{b}, from this we infer that $\psi(t)\cdot b=b$
for every $t$, and therefore, $c=b$\,---\,a contradiction. Thus, $G\cdot  b$ is closed. Since $a_1,\ldots, a_q$ are nonzero, \eqref{G}, \eqref{chi}, \eqref{action}, \eqref{lxxx}, and \eqref{b} imply that an element $(t_1x_1,\ldots, t_nx_n)\in D_n$ lies in $G_b$ if and only if
 $t_1=\ldots=t_q=1$. This proves (i).

The arguments analogous to that used in the proof of Lemma \ref{npm} yield (ii).
\quad $\square$ \renewcommand{\qed}{}\end{proof}

Finally, consider the case where one of the numbers
 $l_1,\ldots,l_n$ is equal to $\pm 1$
(in view of \eqref{chi}, without loss of generality we may assume that it is equal to 1), and all the others are equal to $0$.

\begin{lemma}\label{11}
Let $l_i=1$ and let $l_j=0$ for $j\neq i$. For any $s\in k$ denote by $H(s)$ the hyperplane in $\An$ defined by the equation $x_i+s=0$. Then:
\begin{enumerate}[\hskip 2.2mm \rm (i)]
\item $\bigcup_{j\neq i} H_j$ is the set of points with nontrivial $G$-stabilizer {\rm(}that auto\-ma\-tically has positive dimension{\rm)}.
\item The open subset $H(s)\setminus \bigcup_{j\neq i} H_j$ of $H(s)$ is an $(n-1)$-dimensional $G$-orbit and every $(n-1)$-dimensional $G$-orbit is of this form.
\end{enumerate}
\end{lemma}
\begin{proof} Part (i) immediately follows from  \eqref{action}, \eqref{G}, and \eqref{chi}, and part (ii) follows from (i), the invariance of $H(\alpha)$, and the equality $\dim G=\dim H(\alpha)=n-1$.
\quad $\square$ \renewcommand{\qed}{}\end{proof}

\section{The group $N_{\Autn}(D_n(l_1,\ldots, l_n))$}

First, we shall prove several general statements about normalizers for the actions on arbitrary affine varieties.

\begin{lemma} \label{noal}
Let $X$ be an irreducible affine variety and let $G$ be an algebraic subgroup of  $\Aut X$. Then the following properties are equivalent:
\begin{enumerate}[\hskip 2.2mm \rm(i)]
\item $N_{\Aut X}(G)$ is an algebraic subgroup of  $\Aut X$.
\item The natural action of $N_{\Aut X}(G)$ on $k[X]$ is locally finite.
\end{enumerate}
\end{lemma}

\begin{proof}   (i)$\Rightarrow$(ii) This follows from the fact that the natural action on  $k[X]$ of every algebraic subgroup of
$\Aut X$ is locally finite  (see\;\cite[Prop.\,1.9]{Borel}).

(ii)$\Rightarrow$(i) Assume that (ii) holds.\;Then in  $k[X]$  there is an $N_{\Aut X}(G)$-invariant finite dimensional  $k$-linear subspace  $V$ containing a system of generators of the $k$-algebra $k[X]$.
Hence the homomorphism
\begin{equation*}
\rho\colon N_{\Aut X}(G)\to \GL(V^*)
\end{equation*}
determined by the action of  $N_{\Aut X}(G)$ on $V$ is an embedding.
Consider the $N_{\Aut X}(G)$-equivariant map
\begin{equation*}
\iota\colon X\to V^*,\quad \iota(x)(f):=f(x)\quad\mbox{for every $x\in X$, $f\in V$.}
\end{equation*}
The standard argument (see\,\cite[Prop.\;1.12]{Borel}) shows that
$\rho|_G$ is a morphism of algebraic groups, and
$\iota$ is a closed embedding. Identify
$X$ with $\iota(X)$ by means of $\iota$, and $N_{\Aut X}(G)$ with $\rho(N_{\Aut X}(G))$ by means of $\rho$. Then  $X$ is a closed subvariety of  $V^*$, and $N_{\Aut X}(G)$ and $G$ are the subgroups of  $\GL(V^*)$; besides, $G$ is closed, and
\begin{gather}
N_{\Aut X}(G)\subset N_{\GL(V^*)}(G)\cap {\rm Tran}_{\GL(V^*)}(X, X), \quad\mbox{where}\label{cap}\\
{\rm Tran}_{\GL(V^*)}(X, X):=\{g\in \GL(V^*)\mid g\cdot  X\subset X\}.\label{tran}
\end{gather}

In fact, in the right-hand side of \eqref{tran}  the equality $g\cdot  X= X$ automatically holds: indeed,  $X$ is irreducible and closed in  $V^*$, and $g\in\Aut V^*$ implies that $g\cdot  X$ is a closed subset of  $X$ of the same dimension as  $X$. Hence ${\rm Tran}_{\GL(V^*)}(X, X)$, as well as  $N_{\GL(V^*)}(G)$, is a subgroup of   $\GL(V^*)$, and therefore, the right-hand side of \eqref{cap} is a subgroup of  $\GL(V^*)$. Its elements normalize $G$ an, being restricted to $X$, are the automorphisms of  $X$; whence, they lie in  $N_{\Aut X}(G)$. Therefore,
 \begin{equation}\label{=}
N_{\Aut X}(G)= N_{\GL(V^*)}(G)\cap {\rm Tran}_{\GL(V^*)}(X, X).
\end{equation}

From closedness of  $G$ in $\GL(V^*)$
and that of $X$ in $V^*$ we deduce, respecti\-ve\-ly, that  $N_{\GL(V^*)}(G)$ and ${\rm Tran}_{\GL(V^*)}(X, X)$ are closed in $\GL(V^*)$ (see \cite[Prop.\,1.7]{Borel}). This and \eqref{=} imply that $N_{\Aut X}(G)$ is closed in $\GL(V^*)$. Hence $N_{\Aut X}(G)$ is an algebraic subgroup of $\Aut X$.
\quad $\square$ \renewcommand{\qed}{}\end{proof}


\begin{theorem} \label{fc} Let $X$ be an irreducible affine variety
and let $G$ be a reductive algebraic subgroup in $\Aut X$ such that
\begin{equation}\label{=k}
k[X]^G=k.
\end{equation}
In either of the following cases
 $N_{\Aut X} (G)$ is the algebraic subgroup of
$\Aut X${\rm:}
 \begin{enumerate}[\hskip 2.2mm \rm(i)]
\item $G$ has a fixed point in $X$.
\item $G^0$ is semisimple.
\end{enumerate}
\end{theorem}
\begin{proof} Take $f\in k[X]$. We shall prove that in each of cases (i) and (ii) the $k$-linear
 span of the orbit $N_{\Aut X}(G)\cdot f$ is finite dimensional. The claim of the theorem will then follow from Lemma \ref{noal}.

Let ${\mathcal M}(G)$ be the set of isomorphism classes of algebraic simple $G$-mo\-du\-les. If $L$ is an algebraic  $G$-module, denote by $L_{\mu}$ its isotypic component of type  $\mu\in {\mathcal M}(G)$.

Since $G$ is reductive, we have (see\;\cite[3.13]{PV})
\begin{equation}\label{decomp}
k[X]=\bigoplus_{\mu\in {\mathcal M}(G)} k[X]_\mu
\end{equation}

The group $N_{\Aut X}(G)$ permutes the isotypic components of the $G$-module $k[X]$.

Since $k[X]_\mu$ is a finitely generated $k[X]^G$-module (see\,\cite[Thm. 3.24]{PV}),
\eqref{=k} implies that
\begin{equation}\label{fin}
\dim_k k[X]_\mu<\infty\quad \mbox{for every $\mu$}.
\end{equation}

In view of \eqref{decomp}, there are elements $\mu_1,\ldots, \mu_s\in {\mathcal M}(G)$ such that
\begin{equation}\label{mui}
k[X]_{\mu_i}\neq 0\;\;\mbox{ for all $i$,}\quad\mbox{and}\quad
f\in k[X]_{\mu_1}\oplus\cdots\oplus  k[X]_{\mu_s}.
\end{equation}

(i) Assume that in $X$ there is a $G$-fixed point $a$. Since closed orbits are separated by  $G$-invariant regular functions (see\,\cite[Thm.\,4.7]{PV}), it follows from \eqref{=k} that
there are no other $G$-fixed points in $X$. Hence $a$ is fixed by $N_{\Aut X}(G)$ as well. Therefore, the ideal
\begin{equation*}
\mathfrak m_a:=\{f\in k[X]\mid f(a)=0\}
\end{equation*}
is $N_{\Aut X}(G)$-invariant. Hence every member of the decreasing  filtration
\begin{equation}\label{filt}
\mathfrak m_a\supset
\cdots\supset \mathfrak m_a^d\supset\mathfrak m_a^{d+1}\supset\cdots.
\end{equation}
is $N_{\Aut X}(G)$-invariant. This filtration has the property (see\;\cite[Cor. 10.18]{AM}) that
\begin{equation}\label{zero}
\bigcap_d \mathfrak m_a^d=0.
\end{equation}

In view of \eqref{filt}, we have a decreasing system of
nested linear subspaces
 $\{k[X]_\mu\cap \mathfrak m_a^d\mid d=1,2,\ldots\}$. Since they are finite dimensional
 (see \eqref{fin}), there is $d_\mu$ such that
$k[X]_\mu\cap \mathfrak m_a^d=k[X]_\mu\cap \mathfrak m_a^{d+1}$ for every $d\geqslant d_\mu$. From \eqref{zero} it then follows that, in fact,
\begin{equation}\label{dmu}
k[X]_\mu\cap \mathfrak m_a^d=0\quad \mbox{for every $d\geqslant d_\mu$.}
\end{equation}

 Let $l\in\mathbf Z$, $l\geqslant \max\{d_{\mu_1},\ldots, d_{\mu_s}\}$. From \eqref{dmu} it follows that
\begin{equation}\label{lxxxx}
 k[X]_{\mu_i}\cap \mathfrak m_a^l=0
\quad \mbox{for every $i=1,\ldots, s$}.
\end{equation}

Since the natural projection $\pi\colon k[X]\to k[X]/\mathfrak m_a^l$
is an epimorphism of $G$-modules, we have $\pi(k[X]_\mu)=(k[X]/\mathfrak m_a^l)_\mu$
for every $\mu\in \mathcal M(G)$. Whereas
\begin{equation*}
\dim_k k[X]/\mathfrak m_a^l<\infty,
\end{equation*}
(see \cite[Prop.\,11.4]{AM}), this implies finiteness of the set of $\mu\in \mathcal M(G)$ such that
\begin{equation}\label{all}
 k[X]_{\mu}\neq 0\quad\mbox{и}\quad k[X]_{\mu}\cap \mathfrak m_a^l=0.
\end{equation}
 Let $\{\lambda_1,\ldots,\lambda_t\}$ be this set. Since  \eqref{all} holds for  $\mu=\mu_1,\ldots,\mu_s$ (see \eqref{mui}, \eqref{lxxxx}), we may assume that
\begin{equation}\label{lm}
\lambda_i=\mu_i\quad \mbox{при $i=1,\ldots, s$}.
\end{equation}
Whereas the group $N_{\Aut X}(G)$ permutes the isotypic components of the $G$-module $k[X]$ and sends  $\mathfrak m_a^l$ into itself, it permutes $k[X]_{\lambda_1},\ldots, k[X]_{\lambda_t}$.\;Hence
$k[X]_{\lambda_1}\oplus\cdots\oplus k[X]_{\lambda_t}$ is an
$N_{\Autn}(G)$-invariant
subspace of  $k[X]$.
 In views of \eqref{fin}, \eqref{mui}, and \eqref{lm}, it is finite dimensional and contains $f$. Hence the $k$-linear span of the orbit $N_{\Aut X}(G)\cdot f$ is finite dimensional,
 as claimed.

(ii) Let  $G^0$ be semisimple. From the Weyl formula for dimension of simple $G^0$-module and finiteness of the index $[G\,{:}\,G^0]$ it follows that, up to iso\-mor\-phism, there are only finitely many
algebraic simple $G$-modules whose dimension does not exceed any preassigned constant. This implies finiteness of the set of $\mu\in \mathcal M(G)$ such that
\begin{equation*}
k[X]_\mu\neq 0 \quad\mbox{и}\quad \dim_k k[X]_\mu \leqslant \max_i \dim_k k[X]_{\mu_i}.
\end{equation*}
Let $\{\lambda_1,\ldots,\lambda_t\}$ be this set.\,We may assume that \eqref{lm} holds.\,Since $N_{\Aut X}(G)$, permuting isotypic components, preserves their dimensions, $k[X]_{\lambda_1}\oplus\cdots\oplus k[X]_{\lambda_t}$ is invariant with respect to $N_{\Aut X}(G)$. The proof can be now comple\-ted as in case (i).
\quad $\square$ \renewcommand{\qed}{}\end{proof}

\begin{corollary}\label{sl} Let  $X$ and $G$ be the same as in Theorem {\rm \ref{fc}} and let $k=\mathbf C$. Assume that the variety $X$ is simply connected and smooth, and $\chi(X)=1$. Then
 $N_{\Aut X} (G)$ is an algebraic subgroup of $\Aut X$.
\end{corollary}

\begin{proof}
By corollary of theorem on \'etale slice (see \cite[III, Cor.\,2]{Luna} and \cite[Thm.\,6.7]{PV}), condition \eqref{=k} and smoothness of $X$ imply that $X$ is a homogeneous vector bundle over the unique closed $G$-orbit $\mathcal O$ in $X$. Hence $\mathcal O$ is simply connected and $\chi(X)=\chi (\mathcal O)$. Being affine, $\mathcal O$ is isomorphic to $G/H$ for some reductive subgroup
 $H$ (see \cite[Thm.\;4.17]{PV}).
 The conditions that $G/H$ is simply connected and $\chi(G/H)=1$ imply that
 $G=H$
(see\,\cite[5.1]{KP}).
 Hence $\mathcal O$ is a fixed point. The claim now follows from Theorem \ref{fc}(i).
\quad $\square$ \renewcommand{\qed}{}\end{proof}

\begin{corollary}\label{sll} Let $G$ be an reductive algebraic subgroups of  $\Autn$ and
$k[\An]^G=k$. Then $N_{\Autn}(G)$ is the algebraic subgroup of $\Autn$.
\end{corollary}

\begin{proof} In view of ${\rm char}\,k=0$, by the Lefschetz principle \cite[15.1]{Har} we may assume that $k=\mathbf C$. Since
$\An$ is simply connected and smooth, and  $\chi(\An)=1$,
the claim follows from Corollary \ref{sl}.
\quad $\square$ \renewcommand{\qed}{}\end{proof}

\begin{remark} {\rm The following example shows that condition
\eqref{=k} alone (for irre\-du\-cible affine  $X$ and reductive и  $G$) does not, in general,
imply that
 $N\!_{\Aut X}\!(G)$ is algebraic.}

\begin{example} Let $G$ be an algebraic torus of dimension $n\geqslant 2$. Take as  $X$ the underlying variety of the algebraic group $G$. Its group of automorphisms $\Aut\!_{\rm gr} G$ is embedded in $\Aut X$ and is isomorphic to $\GL_n(\mathbf Z)$. The action of $G$ on $X$ by left translations embeds $G$ in $\Aut X$. These two subgroups generate $\Aut X$, more precisely, $\Aut X=\Aut_{\rm gr} G \ltimes G$. Therefore, $N_{\Aut X}(G)=\Aut X$. Let $g\in \Aut_{\rm gr} G$ be an element of infinite order and let  $f_1,\ldots, f_n\in k[X]$ be a base of ${\rm X}(G)$. Then $g^d\cdot  f_i\in {\rm X}(G)$ for every $d\in \mathbf Z$ and $i=1,\ldots,n$, and the set $C_i:=\{ g^d\cdot  f_i \mid i\in\mathbf Z\}$ is finite if and only if the stabilizer of $f_i$ with respect to the cyclic group generated by $g$ is nontrivial. Assume that all $C_1,\ldots, C_n$ are finite. Then there is  $d\in \mathbf Z$, $d\neq 0$, such that $g^d\cdot  f_i=f_i$ for every
$i=1,\ldots, n$. Since $f_1,\ldots, f_n$ is a base in ${\rm X}(G)$, this means that the autormorphism $g^d$ is trivial contrary to the assumption that the order of  $g$ is infinite and $d\neq 0$. Hence $C_i$ is infinite for some $i$. Since different characters are linear independent over $k$ (see \cite[Lemma 8.1]{Borel}), this implies that the $k$-linear span of the set $C_i$ (and hence of the orbit $N_{\Aut X}(G)\cdot f_i$) is infinite dimensional.
\end{example}
\end{remark}

We shall now use the obtained information in the proof of algebraicity of the groups  $N_{\Autn}(D_n(l_1,\ldots,l_n))$.

\begin{theorem}\label{normal}
$N_{\Autn}(D_n(l_1,\ldots,l_n))$ for every $l_1,\ldots,l_n$ is an algebraic sub\-group of  $\Autn$.
Moreover,
\begin{enumerate}[\hskip 2.2mm\rm(i)]
\item $N_{\Autn}(D_n)=N_{\GL_n}(D_n)$.
\item If $(l_1,\ldots,l_n)\neq (0,\ldots,0)$, then
\begin{equation}\label{subs}
N_{\Autn}(D_n(l_1,\ldots,l_n))\subseteq N_{\GL_n}(D_n)
 \end{equation}
in either of the following cases:
\begin{enumerate}[\hskip 2.2mm\rm(a)]
\item  all the numbers $l_1,\ldots, l_n$ are nonzero and have the same sign;
\item   none of the numbers $l_1,\ldots, l_n$ is equal to $\pm 1$;
\item the numbers $l_1,\ldots, l_n$ contain $0$ and $\pm 1$, at least two of them are nonzero, and all nonzero of them have the same sign.
\end{enumerate}
\item If\;$\;l_i=1$ and $l_j=0$ for $j\neq i$, then $N_{\Autn}(D_n(l_1,\ldots,l_n))$
is isomorphic to $N_{\GL_{n-1}}(D_{n-1})\times {\rm Aff}_1$ and consists of all
$(g_1,\ldots, g_n)\in {\rm Aff}_n$ of the form
\begin{equation}\label{cases}
g_j=\begin{cases} t_jx_{\sigma(j)} & \mbox{\it при $j\neq i$};\\
t_jx_j+s & \mbox{\it при $j= i$,}
\end{cases}
\end{equation}
where  $t_1,\ldots, t_n, s\!\in\! k$ and $\sigma$ is a permutation of the set  $\{1,\!\ldots\!, i-1, i+1,\!\ldots\!, n\}$.
\end{enumerate}
\end{theorem}

\begin{proof}
If $G$ is a subgroup of $\Autn$ and $g\in N_{\Autn}(G)$, $a\in\An$, then
\begin{equation}\label{orb}
\begin{gathered}
g(G\cdot a)=G\cdot g(a),\quad gG_ag^{-1}=G_{g(a)},\quad g(\overline{G\cdot a})=\overline{g(G\cdot a)},\quad\mbox{and}\\[-2pt]
 \mbox{if $G$ is algebraic,}\quad \dim G\cdot a=\dim g(G\cdot a).
\end{gathered}
\end{equation}

Let $G=D_n$.
Lemma \ref{op}(i) implies that $\bigcup_{i=1}^nH_i$ is the set of points whose $G$-stabilizer has positive dimension. By \eqref{orb}, it is $g$-invariant.\;Since the restriction of  $g$ to the variety $\bigcup_{i=1}^nH_i$ is its automorphism, $g$ permutes  its irreducible components  $H_1,\ldots, H_n$,
i.e., there is a permutation
 $\sigma\in S_n$ such that
\begin{equation}\label{sH}
g(H_i)=H_{\sigma(i)}\quad\mbox{for every $i=1,\ldots, n$.}
\end{equation}
 Since the ideal in $k[\An]$ determined by $H_i$ is generated by $x_i$, this shows that the polynomial $g^*(x_i)$ divides $x_{\sigma(i)}$, hence $g^*(x_i)=t_i x_{\sigma(i)}$ for some
nonzero $t_i\in k$. Therefore, $g\in N_{\GL_n}(D_n)$ (see \eqref{g*}, \eqref{torusDn}). This proves\;(i).

Further, let $G=D_n(l_1,\ldots,l_n)$, $(l_1,\ldots,l_n)\neq (0,\ldots, 0)$.

Assume that condition (a) holds. Then \eqref{orb} and Corollary of Lemma \ref{stable} imply that  $g$ permutes   the orbits $\mathcal O_1,\ldots \mathcal O_n$, i.e., there exists a permutation $\sigma\in S_n$ such that $g(\mathcal O_i)=\mathcal O_{\sigma(i)}$ for every $i$. In view of Lemma \ref{i0} and \eqref{orb}, this implies that $g$ has property \eqref{sH}, and as is shown above, the latter implies that $g\in N_{\GL_n}(D_n)$. Thus, \eqref{subs} is proved in the case when condition (a) holds.

Assume that condition (b) holds.\;Then
 $g$-invariance of $\bigcup_{i=1}^nH_i$ follows from \eqref{orb} and Lemma \ref{npm}. Now the same argument as for
$G=D_n$ completes the proof of \eqref{subs} in the case when condition (b) holds.

 Assume that condition (c) holds.\;Proving \eqref{subs}, we may replace $G$ by the group conjugate to $G$ by means of an appropriate element of  $N_{\GL_n}(D_n)$ and assume that \eqref{lxxx} holds.\;The set
$\{a\in\An\mid \dim G_a>0\}$ is closed (see\,\cite[1.4]{PV}).\;By virtue of \eqref{orb}, it is $g$-invariant, and Lemmas \ref{op} and \ref{10}(ii) imply that its $(n-1)$-dimensional irreducible components  are $H_{q+1},\ldots\break\ldots, H_n$. Hence $g$ permutes
$H_{q+1},\ldots, H_n$. Further, it follows from Lemmas \ref{op} and \ref{10}(ii) that $\mathcal O_{p+1},\ldots, \mathcal O_{q}$ are all the $G$-orbits $\mathcal O$ in $\An$ such that $G_a$ is finite and nontrivial for $a\in\mathcal O$. From \eqref{orb} it then follows that
$g$ permutes   the orbits $\mathcal O_{p+1},\ldots, \mathcal O_{q}$ and, therefore, permutes their closures  $H_{p+1},\ldots,  H_{q}$. Finally, in view of Lemmas  \ref{op} and  \ref{10}(ii), all the  $G$-orbits $\mathcal O$ in $\An$ such that $G_a$ is trivial for $a\in \mathcal O$ are exhausted by the orbits  $\mathcal O_1,\ldots, \mathcal O_p$  and  $G\cdot a$,
where $a\notin \bigcup_{i=1}^nH_i$. Since $G$ is reductive, every  $G$-orbit in  $\An$ contains in its closure a unique closed  $G$-orbit (see\,\cite[Cor.\,p.\,189]{PV}). By Lemma \ref{i0}, for each of the orbits $\mathcal O_1,\ldots, \mathcal O_p$ this closed orbit is the fixed point  $(0,\ldots,0)$.
On the other hand, by Lemma  \ref{10}(i),  the closed orbit lying in $\overline{G\cdot a}$
for $a\notin \bigcup_{i=1}^nH_i$ has dimension $q-1\geqslant 1$ and, therefore, is not the fixed point. Hence $g$ permutes $\mathcal O_1,\ldots, \mathcal O_p$, and therefore, permutes   their closures $H_1,\ldots, H_p$. Thereby, it is proved that \eqref{sH} holds for a certain permutation $\sigma\in S_n$. As above, this allows to conclude that \eqref{subs} holds in case (c). Thus,  (ii) is proved.

Assume now that  $l_i=1$ and $l_j=0$ for $j\neq i$. From \eqref{orb} and Lemma \ref{11}(i) it follows that the closed set $\bigcup_{j\neq i}H_j$ is invariant with respect to $g$. Hence $g$ permutes its irreducible components
$H_1,\ldots, H_{i-1}, H_{i+1},\ldots, H_n$. As above, from here we conclude that for $j\neq i$ the equality
$g_j=t_jx_{\sigma(j)}$  holds for some  $t_j\in k$ and some permutation
$\sigma$ of the set $\{1,\ldots, i-1, i+1,\ldots, n\}$. Further,  \eqref{orb}
and Lemma \ref{11} imply that $g(\mathcal O_i)$ is an orbit open in a certain hyperplane  $H(c)$ (see the notation in Lemma \ref{11}). Since $\overline{\mathcal O_i}=H_i$, and $x_i$ and
$x_i+c$ are, respectively, the generators of the ideals of hyperplanes  $H_i$ and $H(c)$,
we conclude that $g_i=g^*(x_i)$ differs from $x_i+c$ only by a nonzero constant factor: $g_i=t_ix_i+s$ for some $t_i, s\in k$, $t_i\neq 0$. Thus,
$g$ is of the form \eqref{cases}. Conversely, clearly, every element $g\in\Autn$ of the form \eqref{cases} normalizes $G$.  This proves (iii).

Finally, let us prove the first claim of this theorem. According to
(i), (ii), and (iii), the group $N_{\Autn}(D_n(l_1,\ldots,l_n))$
is algebraic if either
 $(l_1,\ldots,l_n)=(0,\ldots, 0)$, or
$(l_1,\ldots,l_n)\neq (0,\ldots, 0)$ and
any of conditions
 (a), (b), (c) of statement (ii) or condition of statement  (iii) hold.\;The only case not covered by these conditions is that when among  $l_1,\ldots,l_n$ there are numbers with different signs.
However, it follows from Lemma \ref{clozero}  that in this case
there are no nonconstant $D_n(l_1,\ldots,l_n)$-invariant regular functions on
 $\An$. But then by Corollary \ref{sll} of Theorem
\ref{fc} we conclude that  $N_{\Autn}(D_n(l_1,\ldots,l_n))$ is algebraic.
\quad $\square$ \renewcommand{\qed}{}\end{proof}

\section{Fusion theorems for tori in $\Autn$ and $\Autsn$}

Fusion theorems describe subgroups that control fusion of subsets under conjugation. Namely, let
$G$ be an (abstract) groups and let  $H$ be its subgroup. We say that $N_G(H)$ controls fusion of subsets in $H$ under conjugation by the elements of  $G$,  if the following property holds:
\vskip 1mm
\begin{equation}
\ \label{PF}\tag*{(F)}
\end{equation}\
\vskip -23mm\
\begin{quote}
for any subset $S\subseteq H$ and an element $g\in G$ such that
$gSg^{-1}\subseteq H$, there is an element $w\in N_G(H)$, such that $gsg^{-1}=wsw^{-1}$
for every element $s\in S$.
\end{quote}

\vskip 2mm

If property (F) holds for a pair $(G, H)$, then we say that for $H$ in $G$ {\it fusion theorem} holds.
Notice that the action of   $N_G(H)$ on $H$ by conjugation boils down to the action of  the ``Weyl group'' $N_G(H)/Z_G(H)$.

\begin{examples} Fusion theorem for $H$ in $G$ holds in the following cases:

1. $G$ is a finite group and  $H$ is its Abelian Sylow  $p$-subgroup. It is the classical Burnside's result.

2. $G$ is an affine algebraic group and  $H$ is its maximal torus. It is the classical result of the theory of algebraic groups, see, e.g., \cite[1.1.1]{Serre1}.

3. $(G, H)=({\rm Cr}_n, D_n)$. It is Serre's result \cite[Thm.\,1.1]{Serre2}. Since every
 $n$-dimensional torus in  ${\rm Cr}_n$ is maximal and conjugate to $D_n$ (see\;Theorem \ref{max1}(i),(ii)), here
 one can replace
  $D_n$ by any  $n$-dimensional torus.
 Recall that for $n\geqslant 5$ in ${\rm Cr}_n$ there exist $(n-3)$-dimensional maximal tori
(see Corollary \ref{c9}(iii)).

 \begin{question}
 {\it Does fusion theorem hold if  $D_n$ is replaced by such a torus}{\rm?}
 \end{question}
\end{examples}

We shall now prove that fusion theorem holds for $n$-dimensional tori in $\Autn$ and $(n-1)$-dimensional tori in $\Autsn$.

\begin{lemma} \label{conju} For every element $g=(g_1,\ldots, g_n)\in \Autn$ there is an ele\-ment
  $g'\in{\rm SL}_n$ such that if
  $s, gsg^{-1}\in \GL_n$, then
 \begin{equation}\label{lin}
 gsg^{-1}=g's{g'}^{-1}.
 \end{equation}
\end{lemma}
\begin{proof}
Let
$g_i=g_i^{(0)}+g_i^{(1)}+\ldots $\;, where $g_i^{(s)}$ is a form of degree $s$ in
$x_1, \ldots, x_n$. Since ${\rm Jac}(g)\in k$, we have
${\rm Jac}(g)=\det(\partial g_i/\partial x_j|^{}_{x_1=\ldots=x_n=0})$.
But the right-hand side of this equality is equal to
$\det(\partial g_i^{(1)}/\partial x_j)$. Hence
\begin{equation}\label{dg}
g^{(1)}:=\big(g_1^{(1)},\ldots, g_n^{(1)}\big)\in\GL_n.
\end{equation}
The automorphism $g^{(1)}$
is the differential of the automorphism
  $g$ at the point $(0,\ldots, 0)$. From \eqref{dg} it follows that
\begin{equation}\label{diffe}
\begin{split}
{\rm(i)}\; &g=g^{(1)}\;\mbox{ if $g\in\GL_n$; }\\[-2pt]
{\rm (ii)}\; &(ga)^{(1)}=g^{(1)}a \;\mbox{ and }\; (ag)^{(1)}=ag^{(1)}\;\mbox{ if $a\in\GL_n$. }
\end{split}
\end{equation}

Now let $s$  and $t:=gsg^{-1}$ be the elements of $\GL_n$. In view of \eqref{dg}, we have
$g^{(1)}\in\GL_n$, and
$gs=tg$ and \eqref{diffe} imply that
$g^{(1)}s=tg^{(1)}$. Therefore,
the product of  $g^{(1)}$ by a constant
$\alpha\in k$ such that $\alpha^n\det g^{(1)}=1$
can be taken as $g'$.\quad $\square$ \renewcommand{\qed}{}\end{proof}

\begin{theorem}\label{FFF} The following pairs $(G, H)$ have property {\rm (F):}
\begin{enumerate}[\hskip 4.2mm \rm (i)]
\item $(${\it Fusion theorem for $n$-dimensional tori in $\Autn$}$)$
\begin{equation*}
(G, H)=(\Autn,
\mbox{an $n$-dimensional torus}).
\end{equation*}
 \item $(${\it Fusion theorem for $(n-1)$-dimensional tori in $\Autsn$}$)$
 \begin{equation*}
(G, H)=(\Autsn, \mbox{an $(n-1)$-dimensional torus}).
\end{equation*}
\end{enumerate}
\end{theorem}

\begin{proof}
(i) Let $G=\Autn$, let $H$ be a torus, and $\dim H=n$.\;Then, by Theorem
  \ref{max2} we may assume that
$H$ is a maximal torus in $\GL_n$. Maintain the notation of the formulation of property \ref{PF}.
By Lemma \ref{conju} there exists an element $g'\in
\GL_n$ such that for every element $s\in S$ equality
 \eqref{lin} holds. Therefore,  $g'Sg'^{-1}=gSg^{-1}\subseteq H$. Since $g'\in\GL_n$, this shows that
\begin{equation}\label{T'}
H':=g'^{-1}Hg'
\end{equation}
is another maximal torus in $\GL_n$ containing $S$. The tori $H$ and $H'$ lie in the  (closed \cite[I.1.7]{Borel}) subgroup
 $Z_{\GL_n}(S)$ of the group $\GL_n$ and therefore, are maximal tori of $Z_{\GL_n}(S)$. In view of conjugacy of maximal tori in any affine algebraic group
\cite[IV.11.3]{Borel}, there is an element
\begin{equation}\label{zzz}
z\in Z_{\GL_n}(S),
\end{equation}
such that
\begin{equation}\label{TT}
H'=zHz^{-1}.
\end{equation}
 From \eqref{T'} and \eqref{TT} it follows that $w:=g'z\in
N_{\GL_n}(H)$, and from \eqref{zzz} and \eqref{lin} that
$gsg^{-1}=wsw^{-1}$ for every $s\in S$. This proves (i).

(ii)  The same argument applies in the case where $G=\Autsn$, $H$ is a torus, and $\dim H=n-1$: one only has to replace  $\GL_n$ by ${\rm SL}_n$, the reference to Theorem \ref{max2}
 by the reference to Theorem \ref{max3}, and notice that
by Lemma \ref{conju} the element $g'$ may be taken from ${\rm SL}_n$.
\quad $\square$ \renewcommand{\qed}{}\end{proof}

\section{Applications: The classifications of classes of conjugate subgroups}

In this section, using the obtained results, we derive the classifications specified in the Introduction and, in particular, prove the generalizations to disconnected groups of the Bia{\l}ynicki-Birula's theorems on linearization of actions on $\An$ of tori of dimension
 $\geqslant n-1$.

\begin{theorem}\label{torus0} Let $G$ be an algebraic subgroups of $\Autn$
 such that $G^0$ is a torus.

\begin{enumerate}[\hskip 4.2mm \rm (i)]
\item If $\dim G= n$ or $n-1$,
then there is an element
$g\in\Autn$ such that $gGg^{-1}\subset \GL_n$ and $gG^0g^{-1}\subset D_n$.
    \item If $G\subset \Autsn$ and $\dim G= n-1$,
then there is an element
 $g\in\Autsn$ such that $gGg^{-1}\subset \SL_n$ and $gG^0g^{-1}=D_n^*$ {\rm(}see\,\eqref{Dnnn}{\rm)}.
\end{enumerate}
\end{theorem}
\begin{proof}  (i) By virtue of conjugacy of maximal tori in
 $\GL_n$, it suffices to prove the existence of  $g\in\Autn$
such that $gGg^{-1}\subset \GL_n$.

 The group $G$ is reductive.\;Therefore, by Corollary of Theorem on \'etale slice
 \cite[Cor.\,2, p.\,98]{Luna} the claim holds if  $k[\An]^G=k$ (here only reductivity of $G$ is used, not the stronger condition that $G^0$ is a torus).
Therefore, in what follows we may assume that $k[\An]^G\neq k$. The latter is equivalent to the condition
\begin{equation}\label{neqk}
k[\An]^{G^0}\neq k.
\end{equation}

In view of Theorem \ref{max2} and equality \eqref{all0},
replacing
$G$ by a conjugate group, we may assume that
\begin{equation}\label{G0}
G^0=D_n(l_1,\ldots,l_n)\subset \GL_n.
 \end{equation}
The claim on the existence of  $g$ will be proved if we show that \eqref{G0} and \eqref{neqk} imply the inclusion
\begin{equation}\label{Glin}
G\subset\GL_n.
\end{equation}

Let $F$ be a finite subgroup of  $G$ that intersects every connected component of this group\,---\,such a subgroup exists, see\;\cite[Lemme 5.11]{BS}. Then
\begin{equation}\label{F}
G=FG^0.
\end{equation}
Since
$G^0$ is a normal subgroup of $G$, we have $F\subset N_{\Autn}(D_n(l_1,\ldots,l_n))$.
In view of Theorem \ref{normal} this shows that
\begin{equation}
\label{Fsubs}
F\subset N_{\GL_n}(D_n)\subset \GL_n,
 \end{equation}
if either  $(l_1,\ldots,l_n)=(0,\ldots, 0)$, or
$(l_1,\ldots,l_n)\neq (0,\ldots, 0)$ and
any of conditions
(a), (b), (c) of statement  (ii) of Theorem \ref{normal} holds. By virtue  of \eqref{G0} and \eqref{F}, this proves \eqref{Glin} for the specified $(l_1,\ldots, l_n)$'s. From Lemma \ref{clozero} and  \eqref{neqk} it now follows that it remains to consider  the last possibility for
 $(l_1,\ldots,l_n)$, namely, that where  $l_i=1$ and $l_j=0$ if $j\neq i$ for certain $i$.

Turning to its consideration, take some element $g\in F$. Then, by Theorem \ref{normal}(iii), equalities \eqref{cases} hold. Since $F$ is finite, the order of $g$ is finite.
This implies that in \eqref{cases} we have $s=0$. This and
 \eqref{torusDn} mean that in the case under consideration \eqref{Fsubs} holds as well, and therefore,  \eqref{Glin}, too. This proves\;(i).

(ii) The proof is the same as that of (i) if  $\Autn$ is replaced by $\Autsn$, $\GL_n$ by  $\SL_n=\GL_n\bigcap \Autsn$, $D_n(l_1,\ldots, l_n)$ by
$D_n^*$, the reference to Theo\-rem \ref{max2} by the reference to Theorem \ref{max3}, and it is taken into account that
$D_n^*$ is the maximal torus of $\SL_n$.
\quad $\square$ \renewcommand{\qed}{}\end{proof}

\begin{theorem} [{\rm Classification of $n$-dimensional diagonalizable subgroups
of $\Autn$ up to conjugacy in $\Autn$}]\label{nn} Up to conjugacy in
 $\Autn$, the torus $D_n$ is the unique  $n$-dimensional diagonalizable subgroup of
 $\Autn$.
\end{theorem}

\begin{proof} By virtue of Theorem \ref{torus0} this follows from the fact that in $\GL_n$
every diagonalizable subgroup is conjugate to a subgroup of the torus $D_n$
(see\,\cite[p.\,112, Prop.(d)]{Borel}).
\quad $\square$ \renewcommand{\qed}{}\end{proof}

\begin{theorem} [{{\rm Classification of $(n-1)$-dimensional diagonalizable sub\-groups of  $\Autsn$ up to conjugacy in $\Autsn$}}]\label{n-1*}
Up to conjugacy in $\Autsn$, the torus $D_n^*$
 is the unique $(n-1)$-dimensional diagonalizable subgroup of
 $\Autsn$.
\end{theorem}

\begin{proof} By virtue of Theorem \ref{torus0}, this follows form the fact that in
$\SL_n$ every diagonalizable subgroup is conjugate to a subgroup of the torus  $D_n^*$
(this easily follows from\,\cite[p.\,112, Prop.(d)]{Borel}).
\quad $\square$ \renewcommand{\qed}{}\end{proof}

\begin{theorem}[{{\rm Classification up to conjugacy in  $\Autn$ of maximal
 $n$-dimen\-sional algebraic subgroups $G$ of $\Autn$ such that $G^0$ is a torus}}]\label{nnnnnn}
Up to conjugacy in  $\Autn$, the group $N_{\GL_n}(D_n)$ {\rm(}see {\rm \eqref{torusDn})} is the unique maximal algebraic subgroup of  $\Autn$ whose connected component of identity is an  $n$-dimensional torus.
\end{theorem}

\begin{proof}
By virtue of Theorem \ref{torus0}, this follows from the fact that  $D_n=N_{\GL_n}(D_n)^0$.
\quad $\square$ \renewcommand{\qed}{}\end{proof}

\begin{theorem}[{{\rm Classification up to conjugacy in $\Autsn$ of maximal
 $(n-1)$-\break di\-mensional algebraic subgroups $G$ in $\Autsn$ such that
 $G^0$ is a torus}}]\label{12}
Up to conjugacy in $\Autsn$,
\begin{equation*}
N_{\SL_n}(D_n^*)=N_{\GL_n}(D_n)\cap \SL_n
 \end{equation*}
is the unique maximal algebraic subgroup of
$\Autsn$ whose connected compo\-nent of identity is an $(n-1)$-dimensional torus.
\end{theorem}

\begin{proof}
By virtue of Theorem \ref{torus0}, this follows from the fact that  $D_n^*=N_{\SL_n}(D_n^*)^0$.
\quad $\square$ \renewcommand{\qed}{}\end{proof}

Denote by ${\mathcal L}_n$ the additive monoid of all $(l_1,\ldots, l_n)\in \mathbf Z^n$  such that
\begin{enumerate}[\hskip 7.2mm \rm (a)
]
\item $l_1\leqslant\cdots\leqslant l_n$;
\item $(l_1,\ldots, l_n)\leqslant (-l_n,\ldots, -l_1)$ with respect to the lexicographical order on $\mathbf Z^n$.
\end{enumerate}

\begin{theorem}[{\rm Classification of  $(n-1)$-dimensional diagonalizable sub\-groups in $\Autn$ up to conjugacy in $\Autn$}]\label{n-1} \

\begin{enumerate}[\hskip 4.2mm \rm (i)]
\item Every $(n\!-\!1)$-dimensional diagonalizable sub\-group of the group ${\rm Aut} {\bf A}\!^n$
is conjugate in $\Autn$ to a unique subgroup of the form
 \begin{equation}\label{DDD}
D_n(l_1,\ldots, l_n),\;\;\mbox{где}\;\; (l_1,\ldots,l_n)\in {\mathcal L}_n\setminus \{(0,\ldots, 0)\}.
\end{equation}
\item
Every subgroup {\rm \eqref{DDD}} is an $(n-1)$-dimensional diagonalizable group.
\end{enumerate}
\end{theorem}

\begin{proof} Let $G$ be an $(n-1)$-dimensional diagonalizable
subgroup of $\Autn$. Taking into account Theorem \ref{torus0}(i),
conjugacy in  $\GL_n$ of every diagonalizable subgroup of this group to a subgroup of the torus $D_n$, and Corollary
 \ref{cod1}(ii), we conclude that  $G$ is conjugate in $\Autn$
to a certain subgroup $D_n(l_1,\ldots, l_n)$ with
$(l_1,\ldots, l_n)\neq (0,\ldots, 0)$.\;It follows from the definition of ${\mathcal L}_n$ and Corollary \ref{cod1}(iv) that we may presume that  $(l_1,\ldots, l_n)\in {\mathcal L}_n$. Assume that $G$ is also conjugate in $\Autn$ to a subgroup $D_n(l'_1,\ldots, l'_n)$ with $(l'_1,\ldots, l'_n)\in {\mathcal L}_n$. Then $D_n(l_1,\ldots, l_n)$ and $D_n(l'_1,\ldots, l'_n)$ are conjugate in $\Autn$, and therefore, Theorem \ref{FFF}(i) implies that they are conjugate in
$N_{\GL_n}(D_n)$. In view of Corollary \ref{cod1}(iv) and Definition ${\mathcal L}_n$, this implies that $(l_1,\ldots, l_n)=(l'_1,\ldots, l'_n)$. This proves (i).

Statement (ii) follows from Corollary \ref{cod1}(ii).
\quad $\square$ \renewcommand{\qed}{}\end{proof}

From Theorems \ref{normal}, \ref{nn}, and \ref{n-1} we deduce

\begin{theorem} \label{algnorm}
If $G$ is a diagonalizable subgroup of  $\Autn$ of dimension $\geqslant n-1$, then
$N_{\Autn}(G)$ is an algebraic subgroup of $\Autn$.
\end{theorem}

\begin{remark} {\rm It is easy to see that $n-1$ in Theorem \ref{algnorm} cannot be replaced by a smaller integer.}
\end{remark}

\begin{theorem}[{\rm Classification of $(n-1)$-dimensional diagonalizable subgroups of $\Autn$ up to conjugacy in $\Cr_n$}]\label{Crn-11} \

 \begin{enumerate}[\hskip 2.2mm \rm(i)]
 \item Two diagonalizable $(n-1)$-dimensional subgroups of the group $\Autn$ are conjugate in $\Cr_n$ if and only if they are isomorphic.
 \item
  Every $(n-1)$-dimensional diagonalizable subgroup of the group
$\Autn$ is conjugate in $\Cr_n$ to a unique closed subgroup of  $D_n$ of the form
\begin{equation*}
{\rm ker}\,\varepsilon_n^d, \;\;d\in\mathbf Z.
\end{equation*}
\end{enumerate}
\end{theorem}

\begin{proof} This follows from Theorems \ref{n-1} and \ref{dc} and Corollary \ref{cod1}(i).
\quad $\square$ \renewcommand{\qed}{}\end{proof}

For $n\leqslant 3$, Theorems \ref{max1}(i), \ref{nn}, and \ref{n-1} yield the classification of all tori in $\Autn$  up to conjugacy in $\Autn$ except for one-dimensional tori in ${\rm Aut}\,\mathbf A^3$. The classification of the latter is given below in Theorem  \ref{1-3}.

\begin{theorem}[{{\rm Classification of one-dimensional tori in
${\rm Aut}\,\mathbf A^3$ up to conjugacy in ${\rm Aut}\,\mathbf A^3$}}] \label{1-3}
Every one-dimensional torus in ${\rm Aut}\,\mathbf A^3$ is conjugate in ${\rm Aut}\,\mathbf A^3$ to a unique torus of the form $T(l_1, l_2, l_3)$, where $(l_1, l_2, l_3)\in \mathcal L_3$.
\end{theorem}

\begin{proof} Let $G$ be a one-dimensional torus in ${\rm Aut}\,\mathbf A^3$. By Theorem \ref{max2} it is conjugate in ${\rm Aut}\,\mathbf A^3$ to some torus  $T(l_1, l_2, l_3)$. In view of \eqref{torusDn} and the equality $T(l_1, l_2, l_3)=T(-l_1, -l_2, -l_3)$, we may presume that $(l_1, l_2, l_3)\in\mathcal L_3$. Assume that $G$ is conjugate in ${\rm Aut}\,\mathbf A^3$ to another torus $T(l_1', l_2', l_3')$ with $(l_1', l_2', l_3')\in\mathcal L_3$.\;Then $T(l_1, l_2, l_3)$ and $T(l_1', l_2', l_3')$ are conjugate in ${\rm Aut}\,\mathbf A^3$, and therefore, by Theorem  \ref{FFF}(i), also  in
$N_{\GL_3}(D_3)$.\;From this, Corollary \ref{cod1}(iv), the definition of  ${\mathcal L}_n$, and Lemma  \ref{TTTT}, it is then not difficult to deduce that $(l_1, l_2, l_3)=(l'_1, l_2', l'_n)$.
\quad $\square$ \renewcommand{\qed}{}\end{proof}

\section{Jordan decomposition in ${\rm Cr}_n$.\;Torsion primes for the Cremona groups}

Although the Cremona groups are infinite dimensional (this has a precise meaning, see \cite{Ram}),
the analogies between them and algebraic groups strike the eye:\;they have the Zariski topology, algebraic subgroups, tori, roots, the Weyl groups, $\ldots$
 In \cite[1.2]{Serre2} Serre writes about the analogy
\begin{equation*}
\mbox{``groupe de Cremona de rang $n$ $\leftarrow\hskip -2mm\rightarrow$ groupe semi-simple de rang $n$''.}
\end{equation*}
Below we briefly touch upon two topics demonstrating that these analogies extend further.\;The second of them is intimately related to tori in the Cre\-mo\-na groups.

\vskip 2mm

\noindent{\bf Jordan decomposition in \boldmath $\Cr_n$.}
\vskip 1mm

Let  $X$ be an algebraic variety. ``Algebraic families'' endow
${\rm Bir}\,X$ with the Zariski topology \cite{Ram}, \cite[2]{Blanc10}, \cite[1.6]{Serre2}: a subset of ${\rm Bir}\,X$ is closed if and only if its inverse image for every algebraic family $S\to {\rm Bir}\,X$ is closed. For every algebraic subgroup $G$ in ${\rm Bir}\,X$ and its subset  $Z$, the closures of  $Z$ in this topology and in the Zariski topology of the group
 $G$ coincide. In particular, $G$ is closed in ${\rm Bir}\,X$.

 Let us call an element $g\in {\rm Bir}\,X$ {\it algebraic} if in  ${\Cr}_n$ there is an algebraic subgroup $G$ containing $g$. This is equivalent to the property that the closure of the cyclic group generated by $g$ is an algebraic group.
If $G$ is affine, then for  $g$ the Jordan decomposition in $G$ is defined, see \cite[Chap.\,I, \S4]{Borel}:
 \begin{equation}\label{J}
 g=g_sg_n
 \end{equation}
 In fact, $g_s$ and $g_n$ depend only of $g$, not of the choice of  $G$. Indeed, let  $G'$ be another affine algebraic subgroup of ${\rm Bir}\,X$ containing $g$, and let $g=g_s'g_n'$ be the Jordan decomposition in $G'$. Since $G\cap G'$ is a closed subgroup of  $G$ and of $G'$, there exists the Jordan decomposition $g=g_s''g_n''$ in $G\cap G'$. Applying theorem on behaviour of the Jordan decompositions under homomorphisms (see \cite[Thm. 4.4(4)]{Borel}) to the embeddings $G \hookleftarrow G\cap G'\hookrightarrow G'$, we obtain $g_s=g_s'', g_n=g_n''$, $g_s'=g_s'', g_n'=g_n''$. Hence, if we call
  \eqref{J} the {\it Jordan decomposition in} ${\rm Bir}\,X$, we get the well-defined notion.

 According to  \cite{Mat}, every algebraic subgroup of  $\Cr_n$ is affine. Therefore, in
$\Cr_n$  every algebraic element admits the Jordan decomposition.

Jordan decompositions in algebraic groups have several known properties
 \cite{Ste}, \cite{Borel}, for instance:
\begin{enumerate}[\hskip 4.2mm \rm(a)]
\item Every semisimple element of a connected group lies in its torus.
\item The set of all unipotent elements is closed.
\item The conjugacy class of every semisimple element of a connected re\-duc\-tive group is closed.
\item The closure of the conjugacy class of every element $g$ of a connected reductive group contains $g_s$.
\end{enumerate}

 Recall that $\Cr_n$ is connected and, for $n=1, 2$, simple \cite{Blanc10}, and that $\Autn$, $\Autsn$ are connected and $\Autsn$ is simple \cite{Shaf}, \cite{Sh}.\;It is natural to ask:

  \begin{question}
 {\it Are there analogues or modifications of the mentioned pro\-per\-ties for the groups $\Cr_n$, $\Autn$, and $\Autsn$}{\rm ?}
  \end{question}

For instance, property (a) for $\Cr_n$ holds if $n=1$, but does not hold if  $n>1$.
Indeed, if $n>2$, then by \cite[Thm.\,4.3]{Popov3} in $\Cr_n$ there is a semisimple element of order two that is not contained in any connected algebraic subgroup of the group
$\Cr_n$. If $n=2$, then by virtue of  Theorem \ref{max1}(ii) and Corollary \ref{isoconj}, all the elements of order  $d<\infty$ in $\Cr_n$ that can be included in tori
constitute a single conjugacy class, while, for instance, for even $d$, the set of all elements of order $d$ is the union of infinitely many conjugacy classes \cite{Blanc07}.\;On the other hand,
in the groups ${\rm Aut}\,\mathbf A\!^2$ and ${\rm Aut}^*\mathbf A\!^2$ property (a) holds, because action of every finite group on  $\mathbf A^2$ is linearizable
\cite{I}.
\vskip 2mm

\noindent{\bf Torsion primes for the Cremona groups.}
\vskip 1mm

 Let $G$ be a connected reductive algebraic group and let $p$ be a prime integer. Recall (see\;\cite[1.3]{Serre1} and references therein) that $p$ is called {\it torsion prime for the group} $G$
if in $G$ there is a finite Abelian $p$-subgroup not contained in any torus of the group $G$. The set ${\rm Tors}(G)$ of all torsion primes for the group $G$ admits various interpretations. For instance,  $p\in{\rm Tors}(G)$ for $k=\mathbf C$ if and only if $\bigoplus_i{\rm H}_i(G, \mathbf Z)$ contains an element of order $p$ (this explains the name). Finding ${\rm Tors}(G)$ is reduced to the case where $G$ is simply connected. For every simple $G$ the set ${\rm Tors}(G)$ is explicitly described.

The possibility to speak about tori in groups of the form
 ${\rm Bir}\,X$ where $X$ is an algebraic variety,
makes it possible to replace
 $G$ in the above definition by
${\rm Bir}\,X$ or its subgroup thereby
obtaining the well-defined notion. In particu\-lar, we consider what is obtained when $G$ is replaced by $\Cr_n$, $\Autn$, or $\Autsn$ as the definitions of torsion primes for these groups.
 Denote the sets of these primes respectively by  ${\rm Tors}(\Cr_n)$,
${\rm Tors}(\Autn)$, and ${\rm Tors}(\Autsn)$. It is natural to ask

 \begin{question}
 {\it What are, explicitly, the sets ${\rm Tors}(\Cr_n)$,
${\rm Tors}(\Autn)$, and ${\rm Tors}(\Autsn)$}{\rm ?}
 \end{question}
\noindent(About ${\rm Tors}(\Cr_n)$ this question has been formulated and discussed in talk \cite{Popov3}.)

Since $\Cr_1={\rm PGL}_2$, we have
\begin{equation}\label{Cr1}
{\rm Tors}(\Cr_1)=\{2\}.
\end{equation}

According to \cite{Blanc07}, for every $d=2, 3$, and $5$, in $\Cr_2$ there are
infinitely many conjugacy classes of elements of order  $d$. On the other hand,
as is explained at the end of section on Jordan decomposition, in $\Cr_2$ there is the unique conjugacy class of elements of order  $d$ contained in tori. Hence $2, 3$, and $5$ are torsion primes for the group $\Cr_2$.\;Consider a prime integer $p>5$.\;According to \cite[Thm.\,E]{F}, every element $g\in\Cr_2$ of order $p$ lies in a subgroup isomorphic to  ${\rm Aut}\,{\mathbf P}^2={\rm PGL}_3$.\;Hence $g$ lies in a torus. Finally, according to \cite[Thm.\,B, p.\,146]{BlancThesis}, in $\Cr_2$ there is a unique up to conjugacy noncyclic finite Abelian  $p$-group (it is denoted by 0.mn), and this group is contained in a maximal torus of a subgroup isomorphic
to $({\rm Aut}\,({\mathbf P}^1\times {\mathbf P}^1))^0={\rm PGL}_2\times{\rm PGL}_2$.\;Thus we conclude that the torsion primes for $\Cr_2$ are the same as that for the exceptional simple algebraic group ${\rm E}_8$:
\begin{equation}\label{Cr2}
{\rm Tors}(\Cr_2)=\{2, 3, 5\}.
\end{equation}

From ${\rm Aut}\,\mathbf A\!^1={\rm Aff}_1$ and \cite{I} it follows that every finite subgroup
of  $\Autn$ (respectively, $\Autsn$) for $n\leqslant 2$ is contained in a subgroup isomor\-phic to $\GL_n$ (respectively, $\SL_n$). Therefore,
\begin{equation*}
{\rm Tors}(\Autn)={\rm Tors}(\Autsn)=\{\varnothing\}\;\;\mbox{for $n\leqslant 2$}.
\end{equation*}

For $n\geqslant 3$, there is no comprehensive information about  the sets ${\rm Tors}(\Cr_n\!)$,
${\rm Tors}(\Autn)$, and ${\rm Tors}(\Autsn)$. From \cite[Thm.\,4.3]{Popov3},
\eqref{Cr1}, and \eqref{Cr2} it fol\-lows that
\begin{equation}\label{n2}
{\rm Tors}(\Cr_n)=\{2, \ldots \}\;\;\mbox{ for every $n$}.
\end{equation}
According to \cite{Beau}, in $\Cr_2$ there is a $3$-elementary Abelian
subgroups of rank $3$ (i.e., isomorphic to $(\boldsymbol\mu_d)^3$).\;Since $\Cr_1$ contains a cyclic subgroup of order  $3$, and  the direct product of $\Cr_2$ and $n-2$ copies of $\Cr_1$
can be embedded in $\Cr_n$ for $n\geqslant 3$, this implies that in
$\Cr_n$ there is a  $3$-elementary Abelian subgroup $G$ of rank $n+1$. But from Lemma \ref{sbgps}(i) it follows that for every prime integer $p$, the rank of any elementary Abelian $p$-subgroup of an $r$-dimensional torus is at most $r$. From this and Theorem \ref{max1}(i)
it follows that  $G$ is not contained in a torus of  $\Cr_n$.\;In view of \eqref{Cr2} and \eqref{n2}, this yields
\begin{equation*}
{\rm Tors}(\Cr_n)=\{2, 3, \ldots \} \;\;\mbox{ for any $n\geqslant 2$}.
\end{equation*}

 \begin{question}
 {\it What is the minimal $n$ such that $7$ lies in one of the sets ${\rm Tors}(\Cr_n)$,
${\rm Tors}(\Autn)$, and ${\rm Tors}(\Autsn)$}{\rm ?}
 \end{question}
 \begin{question}
 {\it Are these sets finite}{\rm ?}
 \end{question}


 \end{document}